\documentclass[11pt]{article}%
\usepackage{float} % fix table inside context 
\usepackage{fullpage}
\usepackage{amsfonts}
\usepackage{amsmath}
\usepackage{amssymb}
\usepackage{amsbsy}
\usepackage{amsthm}
\usepackage{mathtools}
\usepackage{epsfig}
\usepackage{graphicx}
\usepackage{colordvi}
\usepackage{graphics}
\usepackage{color}
\usepackage{bm}
\usepackage{verbatim} 
\numberwithin{equation}{section}
\usepackage{tikz}
\usepackage{ulem}

\newtheorem{theorem}{Theorem}[section]
\newtheorem{thm}{Theorem}[section]

\newtheorem{lemma}{Lemma}[section]

\newtheorem{algorithm}[thm]{Algorithm}

\newtheorem{alg}[thm]{Algorithm}
\newtheorem{remark}[thm]{Remark}

\newcommand{\nr}[1]{\ensuremath{\left\|{#1} \right\|}}

\usepackage{algpseudocode}
\newcommand{\f}{\frac}

\newcommand{\be}{\begin{equation}}
\newcommand{\ee}{\end{equation}}
\newcommand{\bea}{\begin{eqnarray}}
\newcommand{\eea}{\end{eqnarray}}
\newcommand{\beas}{\begin{eqnarray*}}
	\newcommand{\eeas}{\end{eqnarray*}}

\newcommand{\vertiii}[1]{{\left\vert\kern-0.25ex\left\vert\kern-0.25ex\left\vert #1
		\right\vert\kern-0.25ex\right\vert\kern-0.25ex\right\vert}}

\newcommand{\normiii}[1]{{\left\vert\kern-0.25ex\left\vert\kern-0.25ex\left\vert #1
		\right\vert\kern-0.25ex\right\vert\kern-0.25ex\right\vert}}

\usepackage{enumerate}

\begin{document}
\title{Analysis of the Picard-Newton iteration for the Navier-Stokes equations: global stability and quadratic convergence}

\author{Sara Pollock \thanks{\small
		Department of Mathematics, University of Florida, Gainesville, FL, 32611, s.pollock@ufl.edu; partially supported by NSF grant DMS-2011519}  
		\and 
	Leo G. Rebholz\thanks{\small
		School of Mathematical and Statistical Sciences, Clemson University, Clemson, SC, 29364, rebholz@clemson.edu; partially supported by NSF grant DMS-2011490.}
	\and		
	Xuemin Tu\thanks{\small
	Department of Mathematics, University of Kansas, Lawrence, KS, 66049, xuemin@ku.edu.}	
		\and
	Mengying Xiao\thanks{\small
		Department of Mathematics and Statistics, University of West Florida, Pensacola, FL, 32514, mxiao@uwf.edu.}	
		}

	\maketitle
	
	\begin{abstract}{
		We analyze and test a simple-to-implement two-step iteration for the incompressible Navier-Stokes equations that consists of first applying the Picard iteration and then applying the Newton iteration to the Picard output.  We prove that this composition of Picard and Newton converges quadratically, and our analysis (which covers both the unique solution and non-unique solution cases) also suggests that this solver has a larger convergence basin than usual Newton because of the improved stability properties of Picard-Newton over Newton.  Numerical tests show that Picard-Newton dramatically outperforms both the Picard and Newton iterations, especially as the Reynolds number increases.   We also consider enhancing the Picard step with Anderson acceleration (AA), and find that the AAPicard-Newton iteration has even better convergence properties on several benchmark test problems.	
		}
\end{abstract}

	\section{Introduction}
	
	The Navier-Stokes equations (NSE) are widely used to model incompressible Newtonian fluid flow, which has applications across the spectrum of science and engineering.  On a domain $\Omega \subset \mathbb{R}^d, d=2,3$, they take the form
	\begin{equation}\label{NS}
		\left\{\begin{aligned}
			-\nu \Delta u+u\cdot\nabla u+ \nabla p&={f} \quad \text{in}~\Omega,\\
			\nabla\cdot {u}&=0\quad \text{in}~\Omega,\\
			{u}&=0 \quad \text{on}~~\partial\Omega,
		\end{aligned}\right.
	\end{equation}
	where $u$ is the velocity of fluid, $p$ is the pressure, $\nu$ is the kinematic viscosity of the fluid,  and ${f}$ is an external forcing term. The parameter $Re:=\frac{1}{\nu}$ represents the Reynolds number.   While our study herein is restricted to nonlinear solvers for the steady system \eqref{NS} with homogenous Dirichlet boundary conditions, the results are extendable to solving the time dependent NSE at a fixed time step in a temporal discretization, as well as nonhomogeneous mixed Dirichlet/Neumann boundary conditions. 
	
NSE solutions are used in simulation processes that can avoid costly physical experiments.  However, solving the NSE can be very challenging, as the NSE are a notoriously hard nonlinear PDE for which analytical solutions are rarely available and numerical solvers often fail.  The purpose of this paper is to develop and analyze a %(seemingly) 
seemingly
new numerical solver that is simple to implement, is free from parameters, and has better convergence properties than existing solvers.

Perhaps the most commonly used nonlinear iteration for solving \eqref{NS}  is  the Picard iteration, which is given by (suppressing the weak formulation and boundary conditions for the moment)
	\begin{align*}
		-\nu \Delta u_{k+1}+u_k\cdot\nabla u_{k+1}+ \nabla p_{k+1}&={f}, \\
		\nabla\cdot {u}_{k+1}&=0.
	\end{align*}
	We will denote this iteration as a fixed point iteration via $u_{k+1} = g_P(u_k)$.
This iteration is known to be globally stable, i.e. bounded for any initial iterate and problem data.  Moreover, for sufficiently small problem data, then Picard is globally convergent with a linear convergence rate \cite{GR86,PRX19}.  

If a good initial guess $u_0$ is available, then a faster alternative to the Picard iteration is the Newton iteration, which is given by \cite{GP83,John16} 
	\begin{align*}
		-\nu \Delta u_{k+1}+u_k\cdot\nabla u_{k+1}+ u_{k+1}\cdot\nabla u_k + \nabla p_{k+1} &={f} + u_k \cdot\nabla u_k, \\
		\nabla\cdot {u}_{k+1}&=0.
	\end{align*}
	We will denote this iteration as a fixed point iteration via $u_{k+1} = g_N(u_k)$.
This iteration will converge quadratically for small data and for a sufficiently close initial guess \cite{GR86,LHRV23}.  A common strategy for solving NSE has been to run a method with larger convergence radius than Newton such as Picard (or Anderson accelerated Picard \cite{EPRX20,PRX19}) or damped Newton, and then switch over to Newton once the iterates get close enough to the solution \cite{John16}.

Over the years, many variations of the Picard and Newton iterations for the NSE have been developed to improve solver robustness, efficiency, or both.
For example, algorithms have been developed that eliminate or reduce the complexity of the linear solves and saddle point problems \cite{GRVZ23,HL09,RVX19,temam} which provide overall more efficient algorithms if they do not increase the total number of nonlinear iterations.  Continuation methods for the NSE can provide convergence for higher Reynolds numbers \cite{dealII95,MR3043640}.  Anderson acceleration (AA) or other extrapolation methods have been shown to improve convergence and robustness \cite{PRX19,PR21}.  Multilevel methods put most of the computations on the coarser levels \cite{DC08,HW08,KLR06,LLP98,LL95,LL96} and can be quite effective, as can defect correction approaches \cite{KLR06,LLP02}, and there is a long list of other techniques can be effective in certain settings \cite{HL09,MVO21}.  Each of these methods have their drawbacks, whether it be complexity of the implementation, sensitivity to parameters, or inefficiency.	

Herein we consider a simple-to-implement solver for the NSE that is a composition of these two well-known methods as a 2 step iteration which we call the Picard-Newton iteration.  At step $k$, the iteration is defined by first performing a Picard iteration, and then second by performing a Newton iteration on the Picard output, i.e. 
\[
u_{k+1} = g_{PN}(u_k) = g_N ( g_P(u_k)).
\]
This solver is intended for use on problems where Newton will not converge, i.e. for higher Reynolds number problems (otherwise there is clearly extra unnecessary work coming from the Picard step), and the main idea is to combine the robustness of Picard with the convergence order of Newton.  We note in particular the simplicity of the method and its implementation.  
%\color{blue}
%\sout{
%Given the vast amount of research done in the last half century with both Picard and Newton iterations for the NSE, 
%it is hard to imagine that this idea of a Picard-Newton iteration has not been tried 
%before.  However, the authors have been unsuccessful in finding it in the literature.   
%}
%[Danger: could be read as scolding the reviewers for not thinking of it themselves]
%\color{black}

The general idea of composing solvers to take advantage of the good properties of each is well established, and is strongly related to the idea of nonlinear preconditioning - indeed here the Picard iteration can be considered a preconditioner for Newton.  An excellent review paper of Brune et al \cite{BKST15} shows how effective composing solvers can be on a variety of problems.  There are also many works on various combinations of the Picard and Newton methods for systems other than NSE, 
%(although seemingly none as we have defined above), %[sufficiently stated at the end of the sentence]]
such as in imaging \cite{Z19}, periodic systems \cite{LR98}, optimization problems \cite{PMSB12}, and saturated flow problems \cite{PP94}, although none seem to use the simple-to-implement strategy we study herein.
	
The purpose of this paper is to analyze and test the Picard-Newton iteration for the NSE.  The analysis is divided into two parts. First in the case of small data (sufficient for NSE solution uniqueness) we prove the Picard-Newton iteration is stable for any initial iterate (usual Newton is not), and quadratically convergent for a sufficiently close initial guess. Second, for general data where NSE solution uniqueness may not hold, we prove quadratic convergence for a sufficiently close initial guess; and, we show that since the Picard step is globally stable, Picard-Newton is less vulnerable to the blowup behavior common for Newton.  Several numerical tests illustrate the effectiveness of Picard-Newton, and in particular we find excellent results for convergence at higher Reynolds numbers for 2D and 3D driven cavity.  Lastly, we consider enhancing the Picard step with $m=1$ AA, and show that AAPicard-Newton an even more robust and efficient solver that Picard-Newton.

This paper is arranged as follows.  In section 2 we set notation and provide mathematical preliminaries for a smoother analysis to follow.  Section 3 analyzes and tests the Picard-Newton iteration, and section 4 considers AAPicard-Newton.  Conclusions and future directions are discussed in section 5.

\section{Mathematical preliminaries}

We consider an open connected set $\Omega$, and denote the natural function spaces for the NSE by 
	\begin{align}
		&Q:=\{v\in {L}^2(\Omega): \int_{\Omega}v\ dx=0\},\\
		& X:=
		\{v\in H^1\left(\Omega\right): v=0~~\text{on}~ \partial\Omega\},\\
		& V:=
		\{v\in X:  (\nabla \cdot v,q)=0\ \forall q\in Q\}.
	\end{align}
The $L^2$ inner product and norm are denoted by $(\cdot,\cdot)$ and $\|\cdot\|$, respectively.  
Without confusion, the notation 
%$\langle \cdot,\cdot\rangle$ 
$(\cdot,\cdot)$ 
is also
 used to represent the duality between $H^{-1}$ and $X$, and $\|\cdot\|_{-1}$ denotes the norm on $H^{-1}$.

%\color{blue}
%[Sara suggests adding the weak for of \eqref{NS} here]
\color{black}
	
Recall that the Poincare inequality holds on $X$: there exists a constant $C_P$ dependent only on the domain satisfying 
\begin{align}\label{eqn:poincare}
\| \phi \| \le C_P \|\nabla \phi \| \mbox{ for all } \phi\in X.
\end{align}
		
We use the skew-symmetric form of the nonlinear term: for all $v,w,z\in X,$
\[
b^*(v,w,z) = (v\cdot\nabla w,z) + \frac12 ((\nabla \cdot v)w,z),
\]
and note that similar results hold as those below will hold with other energy preserving formulations of the nonlinear term including rotational form and EMAC \cite{GS98,CHOR17,OR20}.  Note an important property of $b^*$ is that $b^*(u,v,v)=0$ for all $u,v\in X$. If the first argument of $b^*$ satisfies $\| \nabla \cdot v\|=0$, then skew-symmetry is not needed (and has no effect).  For $v\in V$, this is the case.  However, when using certain finite element subspaces of $X$ and $Q$ (such as Taylor-Hood elements), it may happen that the discretely divergence free space does not yield completely divergence-free functions.  Hence, to be general and cover all cases, we utilize skew-symmetry in our analysis.

The following inequality hold for $b^*$ \cite{laytonBook,temam}: there exists a constant $M$ dependent only on the domain $\Omega$ such that
	\begin{align}
%b(u,w,v)&\leq M \|\nabla u\|\|\nabla w\|\|\nabla v\|. \label{bstarbound}
b^*(v,w,z)&\leq M \|\nabla v\|\|\nabla w\|\|\nabla z\|. \label{bstarbound}
	\end{align}
%where 
%\[
%M := \sup_{v,w,z\in H^1_0(\Omega)} \frac{ b^*(v,w,z)}{\| \nabla v \| \| \nabla w \| \| \nabla z \|}.
%\]

\subsection{NSE preliminaries}

Written in terms of $b^\ast$, 
the weak form of the NSE \eqref{NS} is given by: Find $u\in V$ satisfying 
\begin{equation}
\nu(\nabla u,\nabla v) + b^*(u,u,v) = (f,v)\ \forall v\in V. \label{weakNS}
\end{equation}

It is well known that the weak steady NSE system \eqref{weakNS} admits solutions for any $f\in H^{-1}(\Omega)$ and $\nu>0$ \cite{GR86,laytonBook}, and that any solution to \eqref{NS} or  \eqref{weakNS} satisfies
\begin{equation}
\| \nabla u \| \le \nu^{-1} \| f \|_{-1}  %\cite{GR86,laytonBook}
. \label{nsstab}
\end{equation}
A sufficient condition for uniqueness of solutions is that the data satisfy \cite{GR86,laytonBook}
\[
\alpha:=M\nu^{-2} \|f \|_{H^{-1}}<1.
\]
While $\alpha<1$ is not a necessary condition for uniqueness, it is known that for large enough data (i.e. $\alpha$ large enough) that uniqueness breaks down \cite{laytonBook} and the NSE will admit multiple solutions.

\section{Convergence analysis of the Picard-Newton iteration}

In this section we analyze stability and convergence results for the Picard-Newton iteration, and test it on several benchmark problems.  First, under the NSE uniqueness condition $\alpha <1$, we prove the method is globally stable (i.e. stable for any initial iterate) and for a sufficiently good initial iterate converges quadratically.  For $\alpha>1$, we can no longer prove global stability but only that the Picard step of Picard-Newton is stable.  We are still able to prove quadratic convergence, provided the initial guess is sufficiently close to a nonsingular NSE solution.

We begin by stating the Picard-Newton iteration in weak form.

\begin{algorithm}[Picard-Newton iteration] 
\label{algpn}
The Picard-Newton iteration consists of applying the composition of the Newton and Picard iterations: $g_{PN} = g_N \circ g_P$, i.e.
\begin{enumerate}
\item [Step 1:] Find $\hat u_{k+1} = g_P(u_{k})$ by finding $\hat u_{k+1}\in V$ satisfying
\begin{equation}
\nu(\nabla \hat u_{k+1},\nabla v) + b^*(u_k,\hat u_{k+1},v) = (f,v) \ \forall v\in V. \label{wp}
\end{equation}
\item [Step 2:] Find $u_{k+1} = g_N(\hat u_{k+1})$ by finding $u_{k+1}\in V$ satisfying
\begin{equation}
\nu(\nabla u_{k+1},\nabla v) + b^*(\hat u_{k+1},u_{k+1},v) + b^*(u_{k+1},\hat u_{k+1},v)-  b^*(\hat u_{k+1},\hat u_{k+1},v) = (f,v) \ \forall v\in V. \label{wn}
\end{equation}
\end{enumerate}
\end{algorithm}

\subsection{Analysis for the case of $\alpha<1$}

Here we analyze stability and then convergence of the Picard-Newton method.  We first prove that the method is globally stable, i.e. bounded by data and independent of the initial iterate.

\begin{lemma}
\label{lemma1}
Under the uniqueness condition $\alpha<1$, we have that Picard-Newton iteration is stable: for every $k$, it holds that
\begin{align*}
\| \nabla \hat u_k \| & \le \nu^{-1} \| f \|_{-1},\\
\| \nabla  u_k \| & \le  \frac{ 1 + \alpha \nu^{-1} } { 1-\alpha} \| f \|_{-1}.
\end{align*}
\end{lemma}

\begin{remark}
Global stability for the usual Newton algorithm for NSE does not hold. For 
$\alpha < 1$, the stability of the Picard-Newton algorithm comes from the
Picard step.
%There is seemingly no stability result for the usual Newton algorithm for NSE that holds for any initial guess.  For Picard-Newton, the stability of Picard is able to provide stability to the Newton step as well.
\end{remark}

\begin{proof}
We begin by proving Step 1 is stable.  Taking $v=\hat u_{k+1}$ in \eqref{wp} vanishes the nonlinear term and gives us 
\[
\nu \| \nabla \hat u_{k+1} \|^2 = (f,\hat u_{k+1}) \le \| f\|_{-1} \| \nabla \hat u_{k+1} \|,
\]
which reduces to
\begin{equation}
\| \nabla \hat u_{k+1} \| \le \nu^{-1} \| f \|_{-1}. \label{pstab}
\end{equation}
Note that this holds for any $\alpha>0$.

Next, we consider Step 2 and take $v=u_{k+1}$, which %kills 
eliminates the first nonlinear term and gives
\[
\nu \| \nabla u_{k+1} \|^2= -b^*(u_{k+1},\hat u_{k+1},u_{k+1})+  b^*(\hat u_{k+1},\hat u_{k+1},u_{k+1}) + (f,u_{k+1} ).
\]
Applying the bound \eqref{bstarbound} on the nonlinear terms and handling the forcing term as in the Picard iteration yields
\[
\nu \| \nabla u_{k+1} \|^2= M \| \nabla u_{k+1} \|^2 \|\nabla  \hat u_{k+1} \| + M \| \nabla \hat u_{k+1} \|^2 \| \nabla u_{k+1}\|  + \| f \|_{-1} \| \nabla u_{k+1}\|.
\]
Using the bound \eqref{pstab} and reducing gives us 
\begin{align*}
 \| \nabla u_{k+1} \|^2 & \le  M\nu^{-2} \| f\|_{-1}  \| \nabla u_{k+1} \|^2 + M \nu^{-3} \| f \|_{-1}^2  \| \nabla u_{k+1}\|  + \| f \|_{-1} \| \nabla u_{k+1}\| \\
 & =  \alpha  \| \nabla u_{k+1} \|^2 +  \alpha \nu^{-1} \| f \|_{-1}  \| \nabla u_{k+1}\|  + \| f \|_{-1} \| \nabla u_{k+1}\|,
\end{align*}
and thus since $\alpha<1$ is assumed,
\[
\| \nabla u_{k+1} \| \le  \frac{ 1 + \alpha \nu^{-1} } { 1-\alpha} \| f \|_{-1}.
\]

\end{proof}

We now prove that the Picard-Newton iteration converges quadratically, provided a sufficiently good initial guess.

\begin{theorem} \label{convthm}
Under the uniqueness condition $\alpha<1$, the Picard-Newton algorithm for the steady NSE converges quadratically provided the initial guess is sufficiently close to the NSE solution \eqref{weakNS}.
\end{theorem}
\begin{remark}
The theorem above requires a sufficiently close initial guess $u_0$, and in particular a sufficient condition that comes from the proof is that
\[
\frac{ M\nu^{-1}\alpha^2 }{1-\alpha} \| \nabla (u-u_0) \| <1.
\]
%Obviously closer is better, but we note 
This shows
that as $\alpha$ gets smaller the convergence basin for the initial guess grows (and approaches all of $X$ as $\alpha\rightarrow 0$).

With the typical guess of $u_0=0$, the condition reduces to 
\[
1 > \frac{ M\nu^{-1}\alpha^2 }{1-\alpha} \| \nabla (u - u_0) \| = \frac{ M\nu^{-1}\alpha^2 }{1-\alpha} \| \nabla u \|.
\]
Since $\| \nabla u \| \le \nu^{-1} \| f \|_{-1}$ by \eqref{nsstab}, this sufficient condition is implied by
\[
1>  \frac{ M\nu^{-2} \| f \|_{-1} \alpha^2 }{1-\alpha} = \frac{\alpha^3}{1-\alpha} \iff 1-\alpha > \alpha^3.
 \]
Hence $\alpha<0.682$ is a sufficient condition for Picard-Newton to converge quadratically when the initial guess is $u_0=0$.  In our tests, this sufficient condition is quite pessimistic and we observe convergence even when $\alpha \gg 1$.
\end{remark}

\begin{proof}[Proof of Theorem \ref{convthm}]
Let $u$ be the solution of \eqref{weakNS}, 
%{\color{blue} satisfying 
%\begin{align}
%\label{ns1}
%\nu(\nabla u, \nabla v) + b^*(u,u,v) = \langle f, v\rangle,
%\end{align} for all $v\in V$.} 
and denote $e_k = u - u_k$ and $\hat e_k = u - \hat u_k$.  
Since $u\in V$, we have  $e_k, \hat e_k \in V$ as well.  

To bound the error in Step 1, we 
subtract the Step 1 equation \eqref{wp} from \eqref{weakNS}, which gives
\[
\nu(\nabla \hat e_{k+1},\nabla v) + b^*(u_k,\hat e_{k+1},v) + b^*(e_k,u,v) = 0 \ \forall v\in V. 
\]
Taking $v=\hat e_{k+1}$, the first nonlinear term vanishes and we obtain
\begin{align*}
\nu \| \nabla \hat e_{k+1} \|^2 &= -b^*(e_k,u,\hat e_{k+1})  
\le M \| \nabla e_k \| \| \nabla u \| \| \nabla \hat e_{k+1} \| 
 \le M \nu^{-1} \| f \|_{-1} \| \nabla e_k \|  \| \nabla \hat e_{k+1} \|,
\end{align*}
thanks to the bound \eqref{bstarbound}  on $b^*$ and \eqref{nsstab}.  Multiplying both sides by $\nu^{-1}$, using the definition of $\alpha$ and reducing, we get the bound
\begin{equation}
\| \nabla \hat e_{k+1} \| \le \alpha \| \nabla e_k \|. \label{picerr}
\end{equation}

%\color{magenta}
%Next, subtract \eqref{weakNS} from Step 2 \eqref{wn}, which yields
%Next, add and subtract $b^*(u,u,v)$ to \eqref{weakNS} to get
%\begin{equation}
%\nu(\nabla u,\nabla v) + b^*(u,u,v) + b^*(u,u,v) - b^*(u,u,v) = (f,v)\ \forall v\in V. \label{weakNS2}
%\end{equation}
Next, subtract Step 2 \eqref{wn} from \eqref{weakNS}, term by term, which yields
\begin{align}\label{n1}
\nu(\nabla e_{k+1},\nabla v) 
+ b^*(\hat u_{k+1},e_{k+1},v)
+ b^*(e_{k+1},\hat u_{k+1},v)
- b^*(\hat e_{k+1},\hat e_{k+1},v)= 0 \ \forall v\in V, 
\end{align}
noting that we used the identities $b^*(u,u,v)-b^*(\phi,\chi,v)=b^*(u-\phi,\chi,v)+b^*(u,u-\chi,v)$
and $b^*(u,u,v)-b^*(\phi,\chi,v)=b^*(u,u-\chi,v)+b^*(u-\phi,\chi,v)$ on the three nonlinear term differences.
%\begin{multline}
%\nu(\nabla e_{k+1},\nabla v) 
%+ b^*(\hat u_{k+1},e_{k+1},v)
%+ b^*(\hat e_{k+1},u,v)
%+ b^*(e_{k+1},\hat u_{k+1},v) 
%+ b^*(u,\hat e_{k+1},v) \\
%- b^*(\hat e_{k+1},u,v)  
%- b^*(\hat u_{k+1},\hat e_{k+1},v)= 0 \ \forall v\in V. \label{n1}
%\end{multline}
%Note that the third and sixth terms cancel, the fifth and last terms combine.
Taking $v=e_{k+1}$ vanishes the second term, leaving
%[I claim there were sign errors and it is now correct (and presented without artificially
%creating terms that cance) Please check again.]
\color{black}
\begin{align*}
\nu \| \nabla e_{k+1} \|^2 & =  - b^*(e_{k+1},\hat u_{k+1},e_{k+1}) + b^*(\hat e_{k+1} ,\hat e_{k+1},e_{k+1})  
\\ & 
\le   M \| \nabla e_{k+1} \|^2 \| \nabla \hat u_{k+1} \|  + M \| \nabla \hat e_{k+1} \|^2  \| \nabla e_{k+1} \|,
\end{align*}
with the last step thanks to  \eqref{bstarbound}.  Now using \eqref{pstab} and \eqref{picerr} and the definition of $\alpha$, we obtain
\begin{align*}
\| \nabla e_{k+1} \| & \le \alpha \| \nabla e_{k+1} \| + M\nu^{-1} \| \nabla \hat e_{k+1} \|^2 
%\\ & 
\le \alpha \| \nabla e_{k+1} \| + M\nu^{-1} \alpha^2 \| \nabla e_{k} \|^2,
\end{align*}
which reduces to 
\[
\| \nabla e_{k+1} \| \le \frac{ M\nu^{-1}\alpha^2 }{1-\alpha} \| \nabla e_{k} \|^2
\]
and proves asymptotic quadratic convergence provided the sequence of iterates converges.  A sufficient condition for this is that $e_0$ satisfies
\[
\frac{ M\nu^{-1}\alpha^2 }{1-\alpha} \| \nabla e_0 \| < 1.
\]
This can be satisfied with $u_0$ being sufficiently close to $u$, and how close it needs to be depends on the size of $\alpha$.

\end{proof}

A somewhat sharper analysis than is given above which could allow for larger $\alpha$ comes from a sharper estimate of the nonlinear term.  A key estimate used in the proofs above is \eqref{bstarbound}, which we recall the proof from \cite{GR86,laytonBook}
comes from H\"older and Sobolev embeddings and first shows
\begin{align}\label{eqn:HSE}
b^*(z,w,v) \le M_0 \| z \|^{1/2} \| \nabla z \|^{1/2} \| \nabla w \| \| \nabla v \|,
\end{align}	
and then applies the Poincar\'e inequality \eqref{eqn:poincare} to obtain
\[b^\ast(z,w,v) 
\le M \| \nabla z \| \| \nabla w \| \| \nabla v \|,
\]
%where the Poincare inequality is used in the last step (i.e. $M=C_P^{1/2} M_0$).  
where $M=C_P^{1/2} M_0$.  
Hence to get a sharper estimate, 
%we can back up one step and remove the Poincare inequality and get
we can back up one step to \eqref{eqn:HSE} to obtain
\begin{align}
b^*(z,w,v)&\leq \frac{M}{C_P^{1/2}} \left( \frac{\| z \|}{\|\nabla z\|} \right) ^{1/2}  \|\nabla z\| \|\nabla w\|\|\nabla v\|. \label{bstarboundb}
\end{align}
%If the domain size is normalized to $O(1)$, 
%{\color{red}Doesn't the Poincar\'e constant also depend on the boundary, and the 
%angles of $\partial \Omega$?}
%then there is likely little gain in the $C_P^{1/2}$ denominator of \eqref{bstarboundb} versus \eqref{bstarbound}.  
%However, the term $\left( \frac{\| z \|}{\|\nabla z\|} \right) ^{1/2}$ is where improvement is possible in \eqref{bstarboundb}, if this ratio is less than 1.  While in general $\| z\| \le \| \nabla z \|$ does not hold, for complex flows we do expect $\| u \| \ll \| \nabla u \|$ (e.g. vorticity is more complex than velocity) and in all of our tests the ratio $\left( \frac{\| u \|}{\|\nabla u\|} \right) ^{1/2}$ is never larger than 0.2.  Hence if the iteration is close to the root, then we expect similar size ratios for $z=u$, $u_{k+1}$, $\hat u_{k+1}$, $e_{k+1}$, and $\hat e_{k+1}$.  While this discussion is by no means a proof, it does suggest that in many cases the analysis above can be sharpened and allow for a less restrictive condition on the data.
The main advantage of the bound \eqref{eqn:HSE} over \eqref{bstarboundb} is that
for complex flows we expect $\| u \| \ll \| \nabla u \|$, 
(i.e. vorticity is more complex than velocity) hence the term
$\left( \frac{\| z \|}{\|\nabla z\|} \right) ^{1/2}$ can in practice be small.
In all of our numerical tests herein, the ratio 
$\left( \frac{\| u \|}{\|\nabla u\|} \right) ^{1/2}$ is never larger than 0.2.  
Hence if the iteration is close to the root, then we expect similar size ratios for 
$z=u$, $u_{k+1}$, $\hat u_{k+1}$, $e_{k+1}$, and $\hat e_{k+1}$.  
This  helps explain why we often in practice see convergence behavior in the regime 
where $\alpha > 1$ that agrees with the above analysis even though the restrictive
data condition is not met.
In the next section we consider further what can be shown if $\alpha < 1$ does not hold.
%While this discussion is by no means a proof, it does suggest that in many cases the 
%analysis above can be sharpened and allow for a less restrictive condition on the data.
\color{black}

\subsection{The case of large problem data}

The proofs above for global stability and convergence of the Picard-Newton iteration no longer hold when $\alpha\ge 1$.  Moreover, classical  convergence proofs of the usual Picard iteration and usual Newton iteration (see e.g. \cite{John16,GR86}) also no longer hold.  Still, convergence of the Picard and Newton iterations for larger data has been shown in many applications, and so we consider this important case also.  We are able to prove for any $\alpha>1$ that Picard-Newton is locally quadratically convergent to a nonsingular NSE solution.

%\color{blue}
%We begin by recalling from \cite{laytonBook} that nonsingular solutions of the NSE exist for any $\alpha>0$, and they are isolated.  
We begin by recalling from \cite[Chapter 6.4]{laytonBook} that isolated nonsingular 
solutions of the NSE exist for any $\alpha>0$.
%That is, there exists some positive $r$ so that a ball radius $r$ around the nonsingular solution contains no other NSE solution.  
For an isolated solution, there exists some positive $r$ so that a ball radius $r$ around that solution contains no other NSE solution.  
%\color{black}
Moreover, if $u^*$ is a nonsingular solution, then there exists $\delta>0$ depending on $\Omega,\ \nu, u^*$ such that for any $w\in V$,
\begin{equation}
\sup_{0\ne v\in V} \frac{ \nu(\nabla w,\nabla v) + b^*(u^*,w,v) + b^*(w,u^*,v)}{\| \nabla v \|} \ge \delta \| \nabla w\|. \label{nsenonsingular}
\end{equation}

This property for nonsingular solutions is key to proving the following local convergence result, which holds for any $\alpha>0$.  We note that the proof of the theorem requires the initial guess to be good enough, and as the Reynolds number grows 
%the guess needs to {\color{red} further improve}.  
the basin of convergence gets smaller.
%It does not seem possible to prove a separate stability result which would hold for a larger class of initial guesses than those that imply the convergence result; 
We restrict our stability analysis to this basin of convergence.
%however, the Picard step is globally stable and so at least the first half of the two-step method is stable.  
However, the Picard step is globally stable so if unstable growth occurs
at a Newton step, the growth is automatically controlled by the globally
bounded Picard step that follows.
In other words, since the stability at the half step is guaranteed, 
the Newton step is always applied with an input velocity that is bounded, 
and can not be part of a sequence of iterates that is blowing up.  
%A proof of such stability has remained elusive, unfortunately, although our numerical tests do show a remarkable stability for Picard-Newton and suggest such a stability proof may exist.
\color{black}

\begin{thm} [Local convergence to nonsingular solutions for any problem data]
Let $u$ be a nonsingular steady NSE solution for given data $f,\nu,\Omega$.  Then for an initial guess sufficiently close to $u$, the Picard-Newton iteration converges quadratically to $u$.
\end{thm}

\begin{proof}
We recall from the convergence proof in the $\alpha<1$ case that \eqref{picerr} holds for any $\alpha>0$:
\[
\| \nabla \hat e_{k+1} \| \le \alpha \| \nabla e_k \|.
\]
To analyze the Newton step, we proceed just as for \eqref{n1}, which states
%\begin{multline*}
%\nu(\nabla e_{k+1},\nabla v) 
%+ b^*(\hat u_{k+1},e_{k+1},v)
%+ b^*(\hat e_{k+1},u,v)
%+ b^*(e_{k+1},\hat u_{k+1},v) 
%+ b^*(u,\hat e_{k+1},v) \\
%- b^*(\hat e_{k+1},u,v)  
%- b^*(\hat u_{k+1},\hat e_{k+1},v)= 0 \ \forall v\in V. 
%\end{multline*}
%The third and sixth terms cancel, and the fifth and last terms combine to produce
\[
\nu(\nabla e_{k+1},\nabla v) 
= -b^*(\hat u_{k+1},e_{k+1},v)
- b^*(e_{k+1},\hat u_{k+1},v) 
+ b^*(\hat e_{k+1},\hat e_{k+1},v) \ \forall v\in V. 
\]
%Next, add $b^*(u,e_{k+1},v)$ and $b^*(e_{k+1},u,v)$ to both sides, and then rearrange terms to find that 
Adding $b^*(u,e_{k+1},v)$ and $b^*(e_{k+1},u,v)$ to both sides yields 
\begin{multline*}
\nu(\nabla e_{k+1},\nabla v) 
+ b^*(u,e_{k+1},v)
+ b^*(e_{k+1},u,v)
= \\
 b^*(\hat e_{k+1},e_{k+1},v)
+ b^*(e_{k+1},\hat e_{k+1},v) 
+ b^*(\hat e_{k+1},\hat e_{k+1},v) \ \forall v\in V. 
\end{multline*}
We now divide both  sides by (nonzero) $\| \nabla v\|$, then majorize the right hand side using \eqref{bstarbound} and then take the supremum over nonzero $v\in V$:
%[Sara says: does $\nabla v \ne 0$ imply $v \ne 0$? Otherwise there is an issue] - yes, due to boundary conditions, \| v \|_H1 is equiv to \| \nabla v \| in V.
\[
\sup_{0\ne v\in V} \frac{ \nu(\nabla e_{k+1},\nabla v) 
+ b^*(u,e_{k+1},v)
+ b^*(e_{k+1},u,v)
}{\| \nabla v \|}
\le
2M \| \nabla  \hat e_{k+1} \| \| \nabla  e_{k+1} \| + M \| \nabla \hat e_{k+1} \|^2.
\]
\color{black}
Thus by \eqref{nsenonsingular} since $u$ is nonsingular we can lower bound the left hand side to get 
\begin{align*}
\delta \| \nabla e_{k+1} \| 
& \le 2M \| \nabla  \hat e_{k+1} \| \| \nabla  e_{k+1} \| + M \| \nabla \hat e_{k+1} \|^2 \\
& \le 2M\alpha \| \nabla  e_{k} \| \| \nabla  e_{k+1} \| + M \alpha^2 \| \nabla e_{k} \|^2,
\end{align*}
thanks to \eqref{picerr}.  For $\| \nabla e_k \|$ sufficiently small, this reduces to
\begin{align}\label{eqn:q0}
(\delta - 2M\alpha \| \nabla e_k \|) \| \nabla e_{k+1} \| \le M\alpha^2 \| \nabla e_k \|^2.
\end{align}
%From here, the assumption of a good enough initial guess is sufficient to imply quadratic convergence.  To be precise, the initial guess must satisfy
%\[
%\lambda:= \delta - 2M\alpha \| \nabla (u-u_0) \|>0
%\]
%and 
%\[
%M\alpha^2 \lambda \| \nabla (u-u_0 \| <1.
%\]

From here, the assumption of a good enough initial guess is sufficient to imply quadratic convergence.  %To be precise, for an initial iterate $u_0$ that satisfies
To be precise, define $\lambda_k = \delta - 2M\alpha \| \nabla (u-u_k) \|$, and 
notice that $\lambda_k > \lambda_j$ whenever $\| \nabla e_k\| < \| \nabla e_j\|$. 
Then for an initial iterate $u_0$ that satisfies
\begin{align}\label{eqn:q1}
\lambda_0 >0 ~\text{ and }~
\frac{M\alpha^2}{ \lambda_0} \| \nabla (u-u_0) \| <1,
\end{align}
each subsequent iterate satisfies 
\begin{align}\label{eqn:q2}
\| \nabla e_{k+1} \| \le \f{M \alpha^2}{\lambda_k} \| \nabla e_k\|^2 
                     \le \f{M \alpha^2}{\lambda_0} \| \nabla e_k\|^2 
                     \le \f{M \alpha^2}{\lambda_0} \| \nabla e_0\| \| \nabla e_k\|,
~k \ge 0, 
\end{align}
courtesy of \eqref{eqn:q0},
by which the iteration converges at least linearly, and in fact quadratically. 

\end{proof}

\begin{remark}
More generally, it does not necessarily have to be the initial iterate $u_0$ that is 
sufficiently close to $u$. Rather, if some iterate $u_J$ satisfies \eqref{eqn:q1},
namely
\[
\lambda_J >0 ~\text{ and }~
\frac{M\alpha^2}{ \lambda_J} \| \nabla (u-u_J) \| <1,
\]
then \eqref{eqn:q2} holds for each $k \ge J$ rather than each $k \ge 0$. This is 
particularly important in light of the previously noted boundedness of each Picard step,
as we notice in practice that the iteration may produce a sequence of bounded (at each 
Picard half-step) then non-converging Newton iterates before falling into the basin of 
convergence. This combination of the Picard iterates preventing the Newton iterates from 
sequentially blowing up in the preasymptotic 
regime, followed by the quadratic convergence from the Newton iterates in the asymptotic
regime, gives this two-step iteration an attractive combination of robustness and 
efficiency. 
\end{remark}
\color{black}

\subsection{Numerical tests for Picard-Newton}

We give here results from three numerical tests with Picard-Newton.  They show that the method has good stability properties, converges quadratically, and is able to converge for much higher Reynolds numbers for the driven cavity problems than usual Picard or Newton.  Scott-Vogelius conforming finite elements are used for all of our tests, with velocity-pressure pairs $(X_h,Q_h)=(P_k,P_{k-1}^{disc})\cap (X,Q)$ and $k=d$.  The meshes are constructed as simplexes, with barycenter refinements  (a.k.a. Alfeld split in the Guzm\'an-Neilan vernacular) as the final step of the mesh construction in order to satisfy the LBB condition \cite{arnold:qin:scott:vogelius:2D,Z05_Alfed}.

For the 2D tests, the linear systems that arise at each step of Picard and Newton are solved with a direct solver.  In 3D, we use GMRESM as the outer solver, and to precondition we follow the ideas of \cite{benzi,HR13} and add grad-div stabilization with parameter 100 (which has no effect on solution since we use divergence-free elements) and precondition the Schur complement with the pressure mass matrix.  We perform the inner solves directly.  This is a very effective solver, and typically it converges in less than 5 iterations with tolerance $10^{-8}$ for both Picard and for Newton.  We tested changing the tolerance for Picard linear solvers to $10^{-3}$ (to make the preconditioner less computationally expensive), and it made almost no difference in the convergence of the nonlinear solver.

\subsubsection{Analytical test}

For our first numerical test, we illustrate the theory above in that Picard-Newton converges quadratically and will have a larger convergence basin than usual Newton.  We consider solving the steady NSE with exact solution
\begin{align*}
u_{nse} = \begin{pmatrix}
\cos(x)\sin(y)\\
-\sin(x)\cos(y)
\end{pmatrix}, \ \ p_{nse} = -\frac{1}4 \left( \cos(2x) + \cos(2y)\right) + x+y,
\end{align*}
over a unit square $\Omega$ and with Reynolds number $Re = 1/\nu =1000$.  We calculate $f$ from the NSE and chosen solution and $\nu$.  Nonhomogenous boundary conditions are used in the tests, and the true solution is used as the boundary condition.  Tests are run for the Picard iteration, the Newton iteration, and the proposed Picard-Newton iteration, with varying initial guesses.

\begin{figure}[h!]
\centering
\includegraphics[scale = .15]{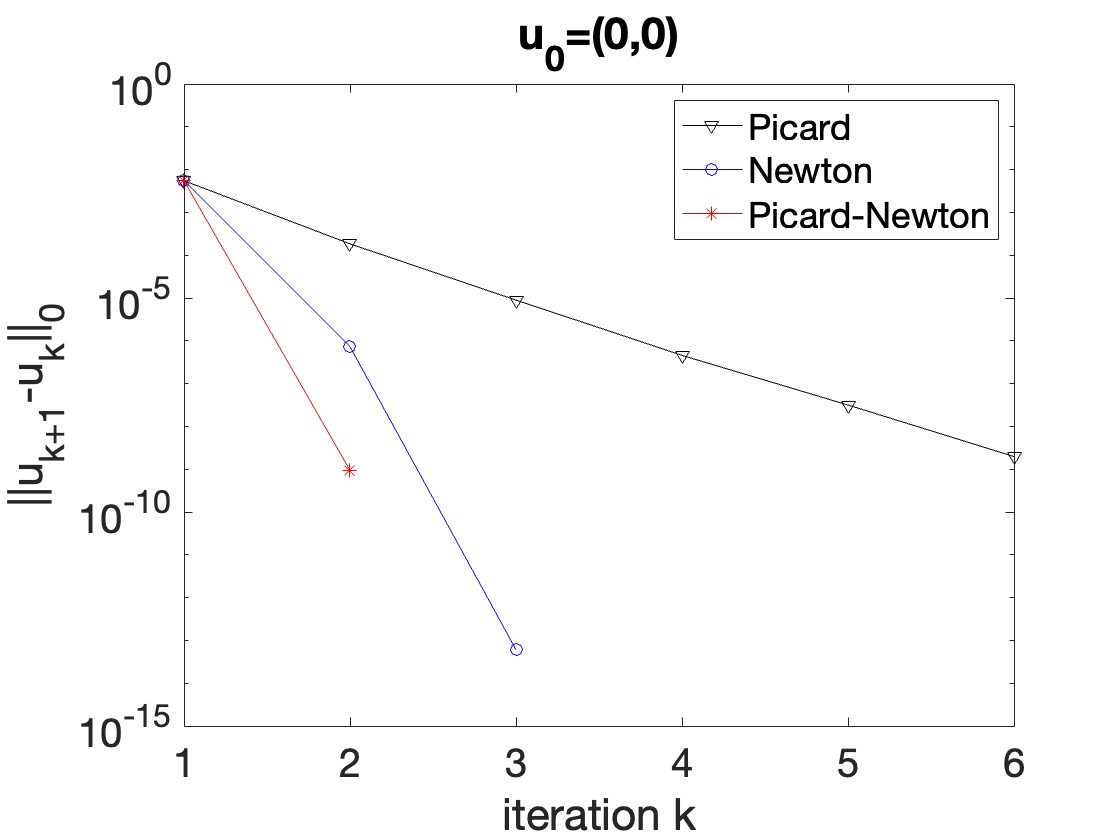}
\includegraphics[scale = .15]{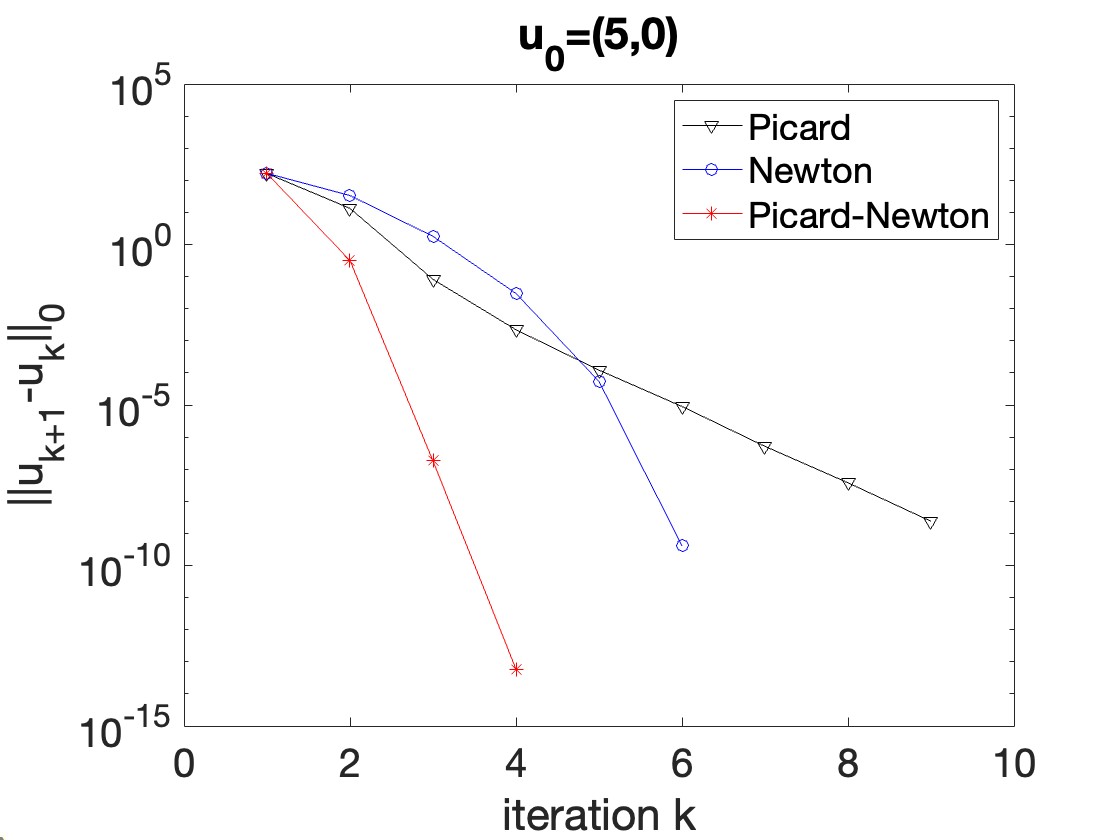}\\
\includegraphics[scale = .15]{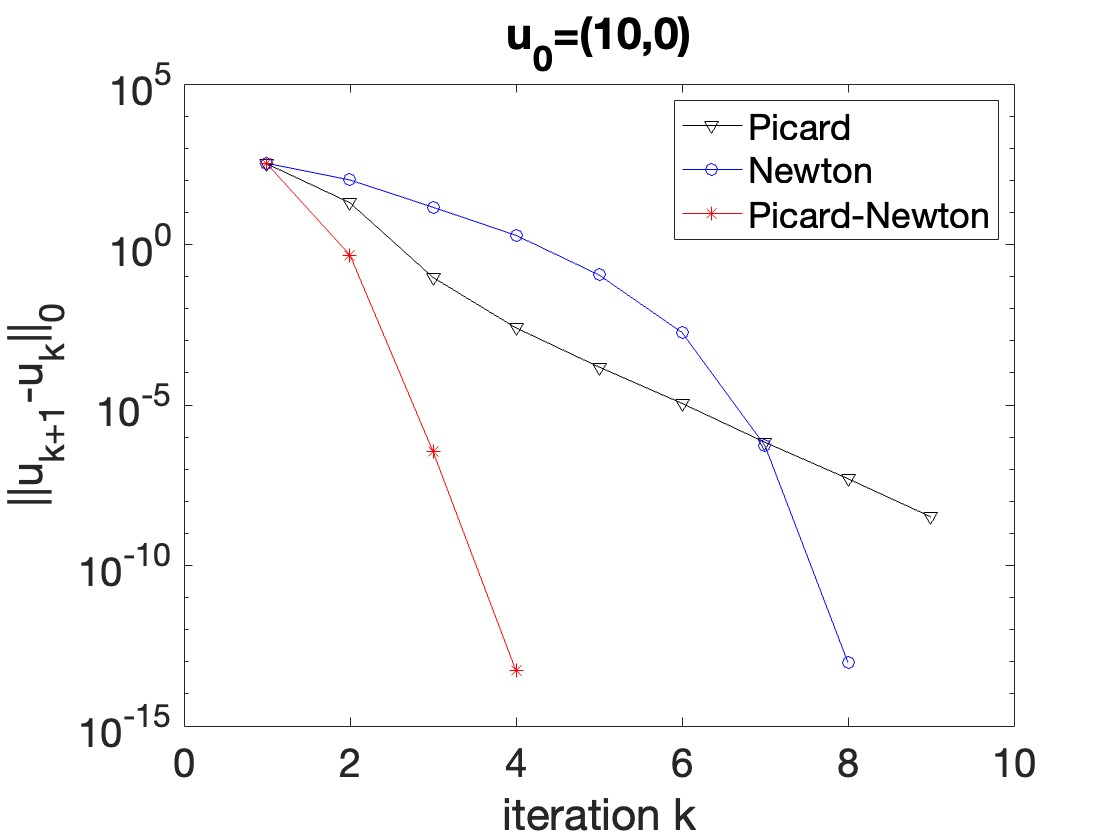}
\includegraphics[scale = .15]{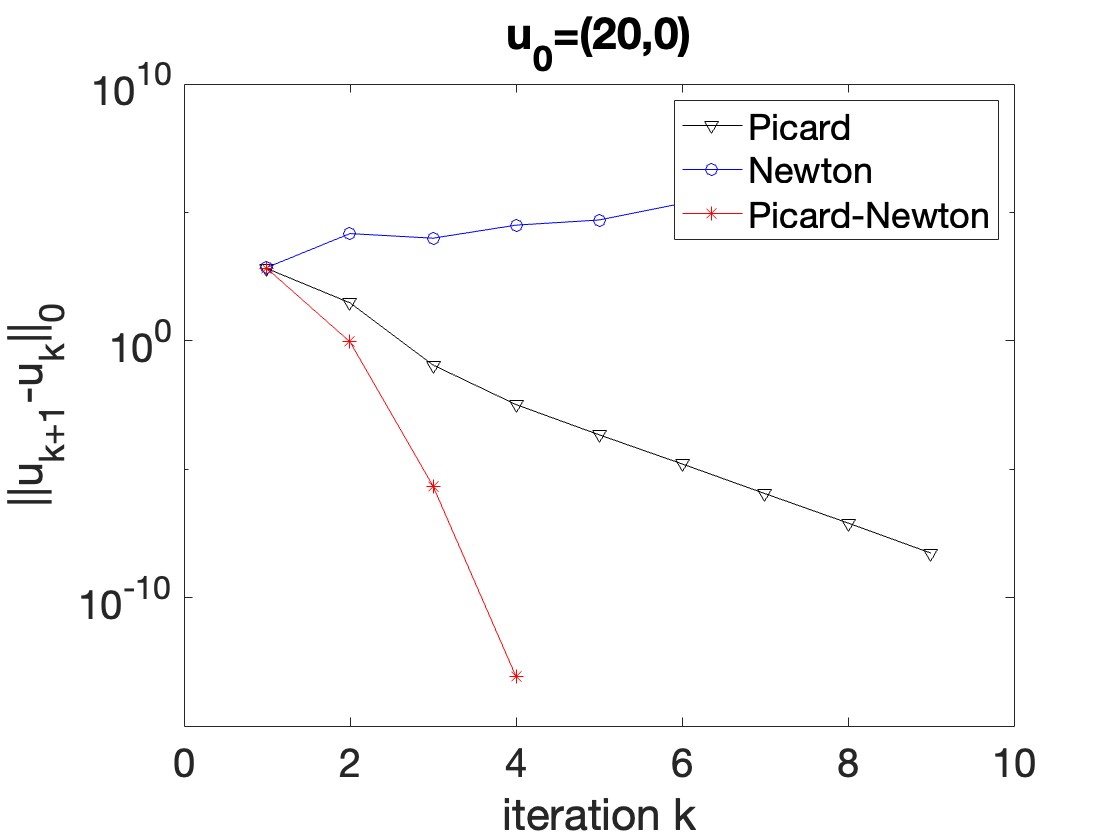}
\caption{Convergence plots of the analytical test with various initial values $u_0 =(c,0)^T, c =0,1,5,10,20$.\label{f1}
}
\end{figure}

For computations, we partition the domain into a barycenter refinement of a uniform triangular mesh with maximum diameter $h = 1/32$,  choose $(P_2, P_1^{dc})$ Scott-Vogelius elements, and set the convergence tolerance for the iterative methods (Picard, Newton and Picard-Newton) to be $tol = 10^{-8}$ in the $L^2$ norm.  \begin{align*}
u_0 = \begin{pmatrix}
c\\0
\end{pmatrix} \text{ in } \Omega, \ u_0 = u_{ex} \text{ on } \partial\Omega,
\end{align*}
with $c=$0, 5, 10, and 20.

Convergence results are shown in figure \ref{f1} for each of the different initial guesses.  We observe linear convergence of Picard in all tests, and superlinear convergence for Newton (when it converged) and Picard-Newton always converging faster than Newton.  However, Newton failed to converge for the largest initial guess.  We note that we checked that all solutions found by the solvers were accurate to within $10^{-4}$ in the $L^2$ norm, which is approximately the optimal accuracy one can expect from interpolation on this mesh.

\subsubsection{2D driven cavity}

For our next test we consider Picard-Newton applied to the 2D driven cavity benchmark problem.  The domain is the unit square $\Omega=(0,1)^2$, and we compute with $(P_2,P_1^{disc})$ Scott-Vogelius elements on a barycenter refined uniform triangular meshes.  We use $f=0$ and Dirichlet boundary conditions that enforce no-slip velocity on the sides and bottom, and $\langle 1,0\rangle^T$ on the top (lid).  In this problem $Re:=\nu^{-1}$.  A plot of the velocity solutions found on the barycenter refined $h=\frac{1}{128}$ mesh (394K velocity degrees of freedom (dof), 294K pressure dof) for varying $Re$ are shown in figure \ref{nseplot1}, and these compare well with the literature \cite{erturk}.    We note that the initial guess is $u_0=0$, and no continuation or acceleration methods are used here.

\begin{figure}[H]
\center
\includegraphics[width = .24\textwidth, height=.24\textwidth,viewport=115 45 460 390, clip]{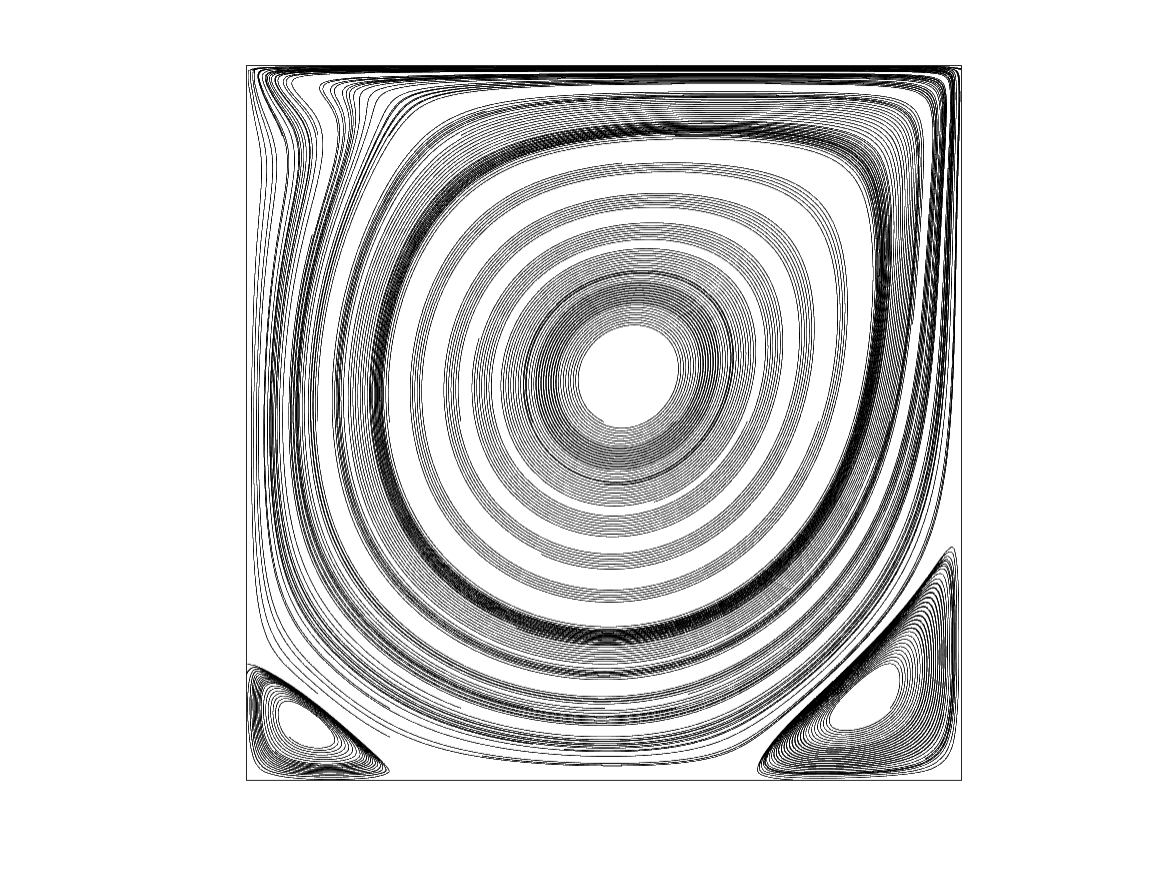}
\includegraphics[width = .24\textwidth, height=.24\textwidth,viewport=115 45 460 390, clip]{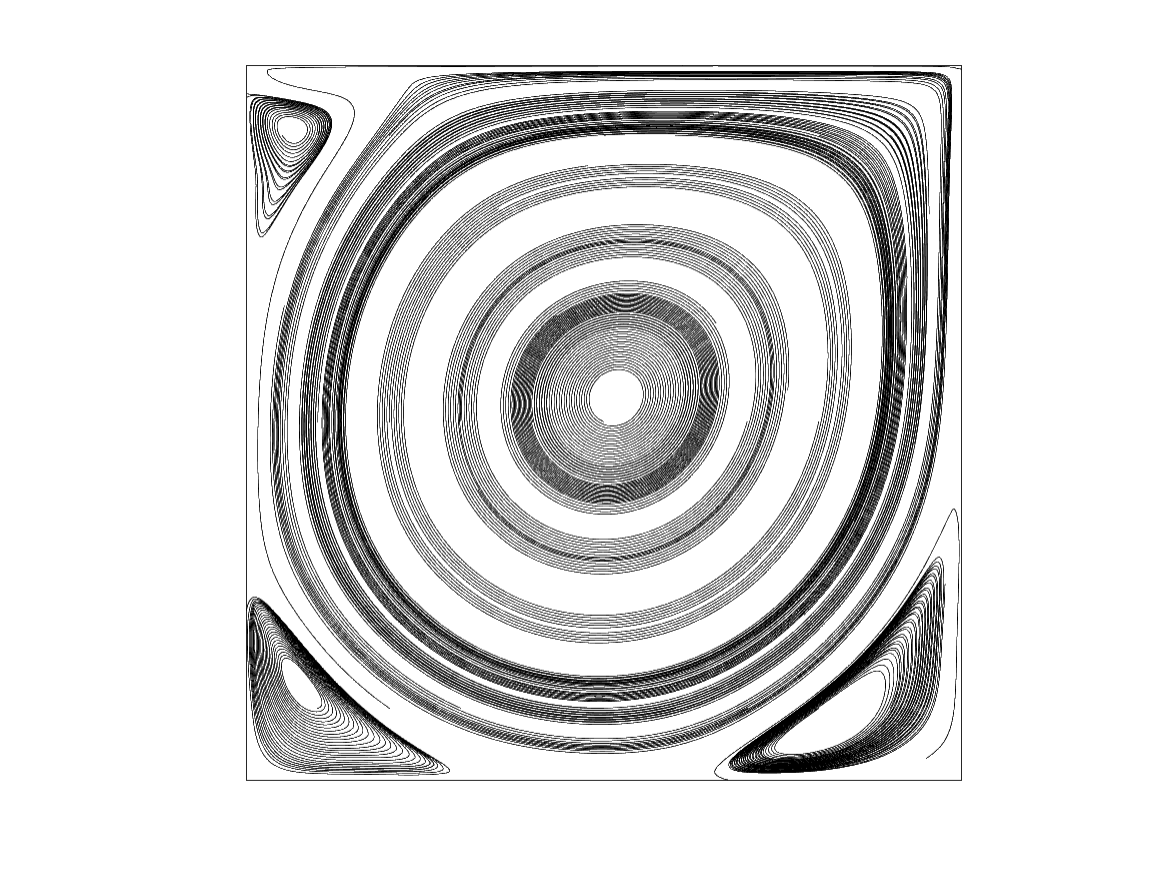}
\includegraphics[width = .24\textwidth, height=.24\textwidth,viewport=115 45 460 390, clip]{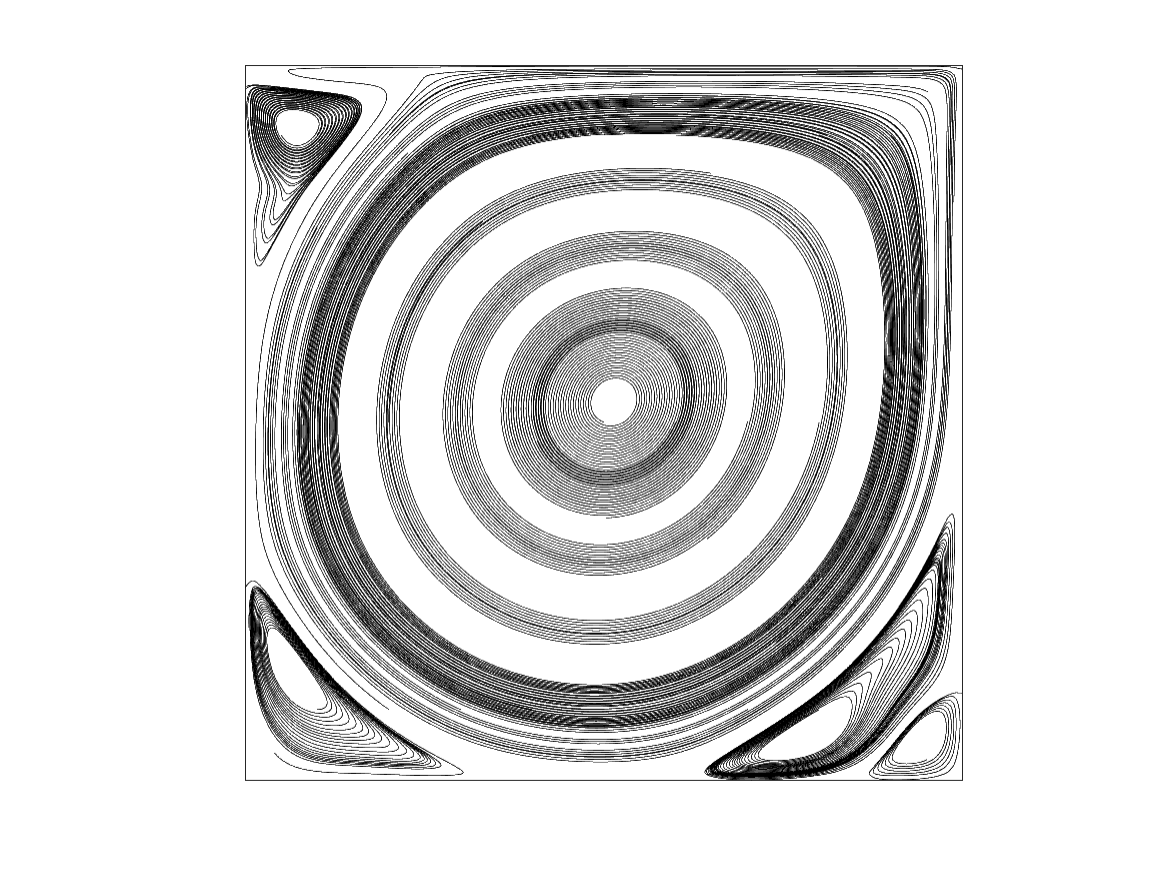}
\includegraphics[width = .24\textwidth, height=.24\textwidth,viewport=115 45 460 390, clip]{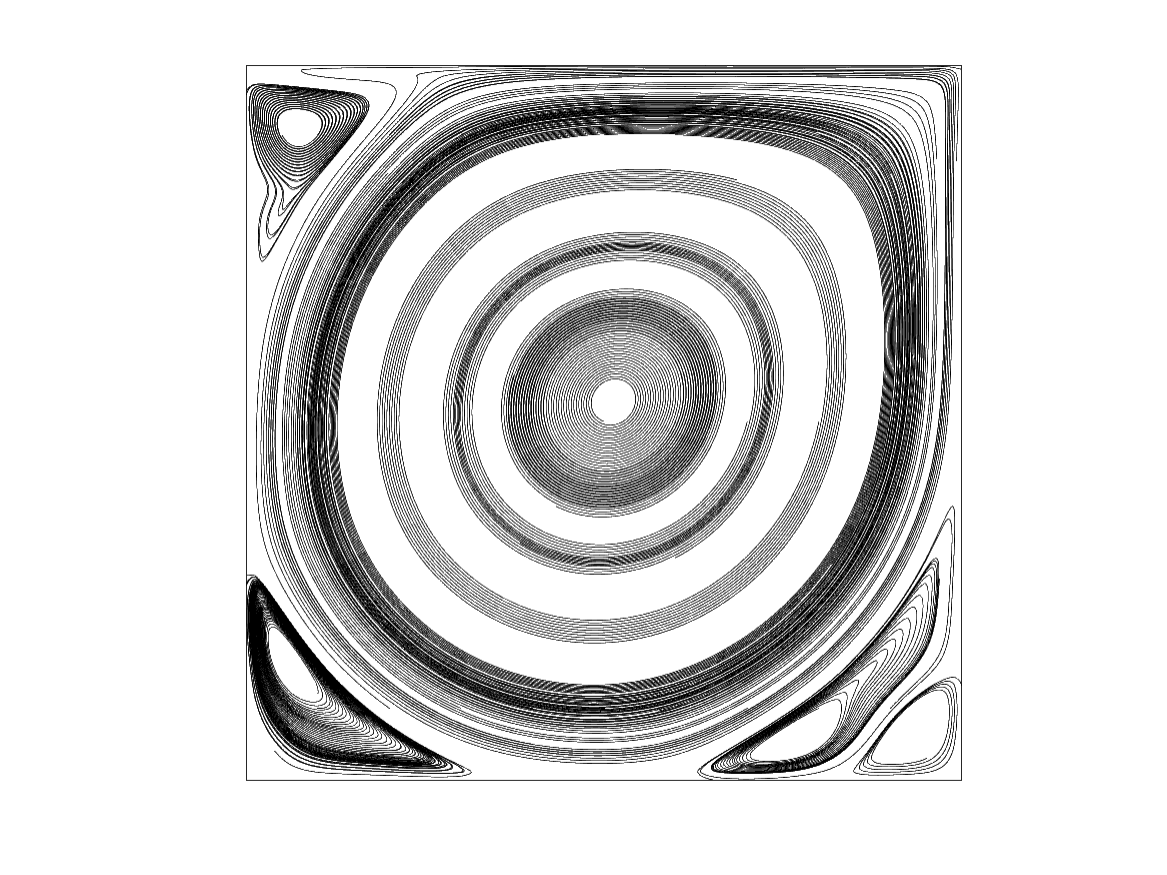}
\caption{\label{nseplot1} Shown above are streamlines of velocity solutions found for the 2D driven cavity problems with $Re$=1000, 5000, 10000, and 12000 (from top left to bottom right).
}
\end{figure}

We run tests with  $Re$=1000, 2500, 5000, 10000 and 12000, on meshes constructed as barycenter refinements of uniform triangulations with
$h=\frac{1}{32},\ \frac{1}{64},\ \frac{1}{128}$ and $\frac{1}{196}$, which provided 42K, 170K, 688K, and 1.6M dof respectively.  Convergence results are shown in figure \ref{conv1} as plots of the fixed point residuals (i.e. difference in successive iterates) in the $H^1$ norm, and we observe quadratic convergence in all cases except for Re=12000 on the finest mesh, which fails.  The convergence for $Re\le 2500$ appears to be mesh independent, which our analysis proves is the case for small data.  However, for $Re\ge 5000$, we observe that convergence is slower on each successive mesh, e.g. for Re=10000 the number of iterations to converge on the successively finer meshes is 6, 23, 52 and 56.  For these higher $Re$, we observe residual values stayed stable but did not decrease in the early iterations, before finally converging.  

\begin{figure}[H]
\center
P-N 1/32+bary \hspace{1.5in} P-N 1/64+bary  \\ 
\includegraphics[width = .4\textwidth, height=.35\textwidth,viewport=0 0 530 400, clip]{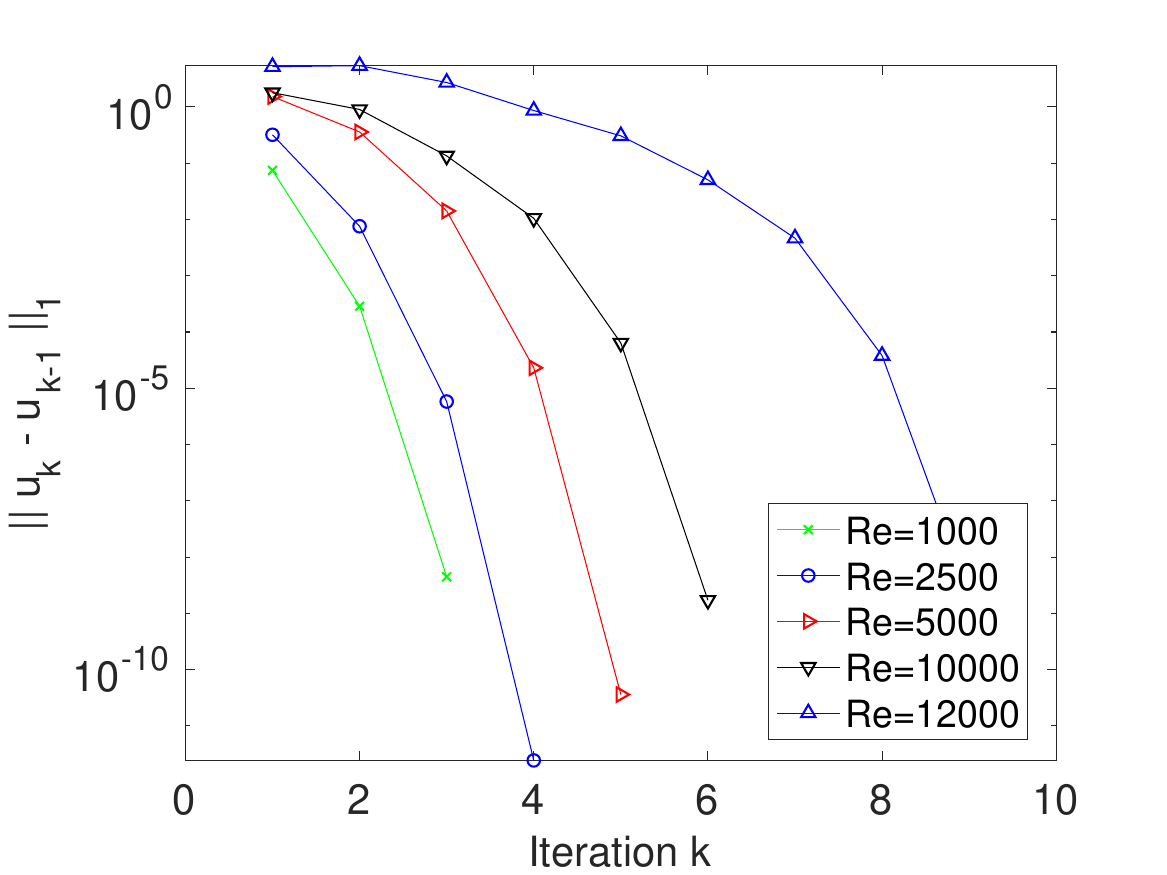}
\includegraphics[width = .4\textwidth, height=.35\textwidth,viewport=0 0 530 400, clip]{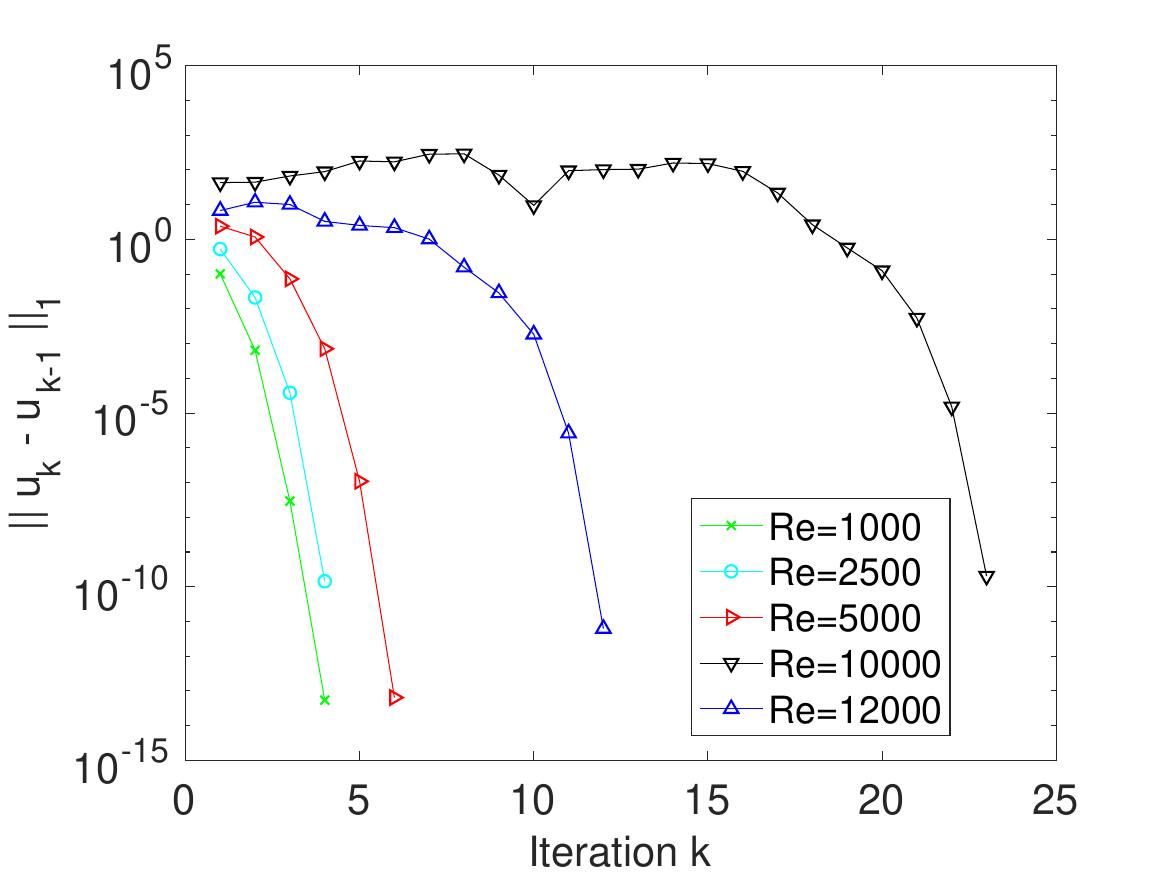} \\
P-N 1/128+bary \hspace{1.5in} P-N 1/196+bary \\
\includegraphics[width = .4\textwidth, height=.35\textwidth,viewport=0 0 530 400, clip]{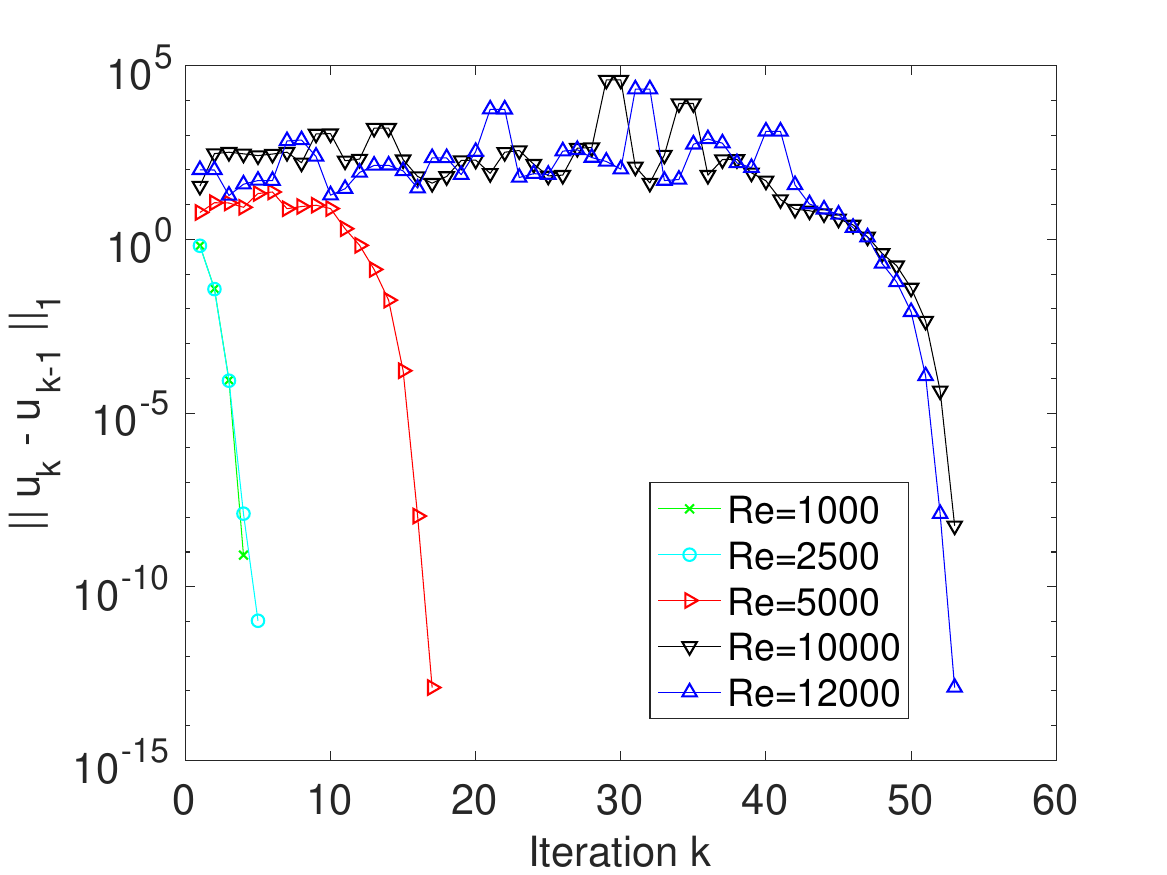}
\includegraphics[width = .4\textwidth, height=.35\textwidth,viewport=0 0 530 400, clip]{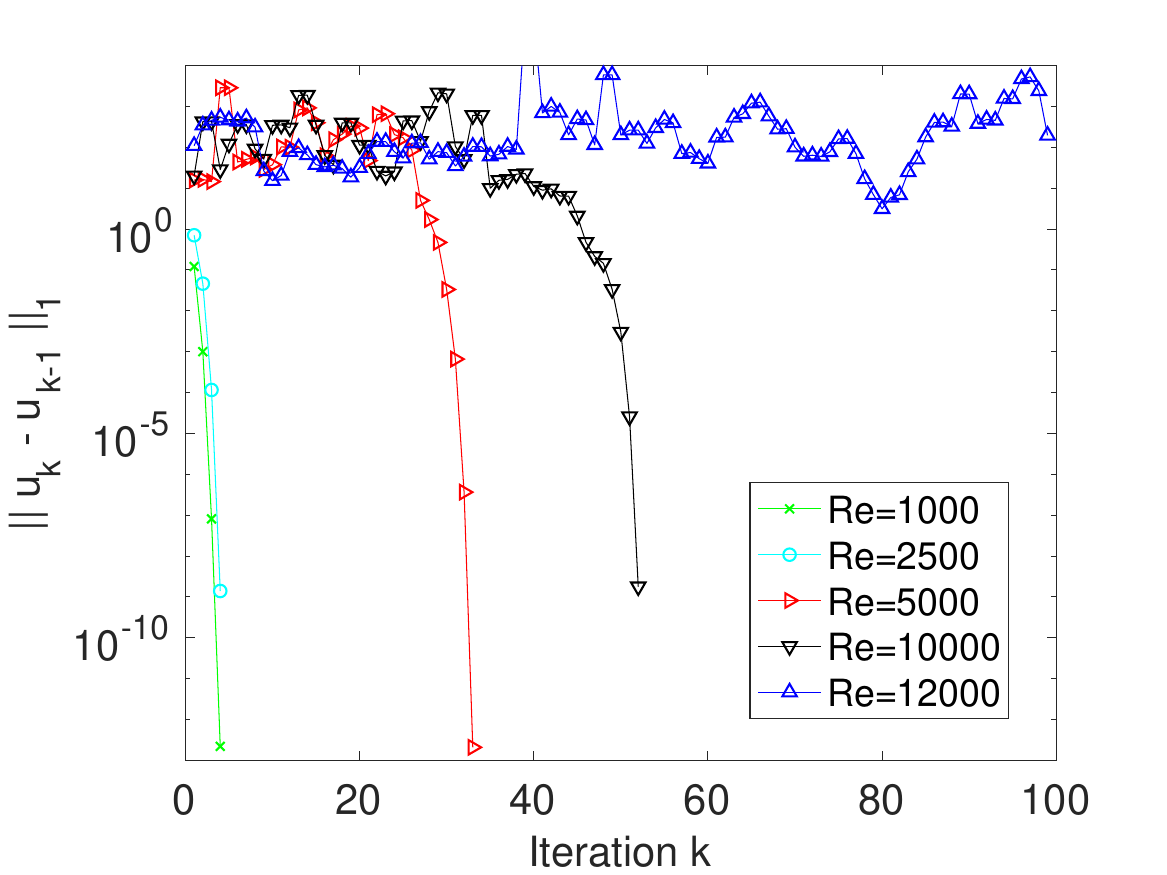}
\caption{\label{conv1} Shown above are convergence plots for the Picard-Newton iteration for varying $Re$, for meshes using  $h=\frac{1}{32}$, $h=\frac{1}{64}$, $h=\frac{1}{128}$ and $h=\frac{1}{196}$ before barycenter refinement.
}
\end{figure}

For comparison, in figure \ref{conv2} we show convergence plots for Picard, Newton and Newton with Line Search (NLS) for varying $Re$ on the barycenter refined $h=\frac{1}{64}$ mesh.  Here we observe that Picard will converge for roughly $Re\le 4000$, Newton will converge for roughly $Re\le 2500$  and Newton with line search will converge for roughly $Re\le 3000$ (the line search implementation checks for a decrease in the nonlinear NSE finite element residual vector for a given step size, and if no decrease is found then the step size is cut in half, to a minimum of $\frac{1}{32}$ step size).  Hence it is clear the proposed Picard-Newton method is able to converge for much higher Reynolds number than Picard, Newton, or Newton with line search are able to.

\begin{figure}[h!]
\center
\hspace{.5in} Picard \hspace{2in} Newton \\
\includegraphics[width = .4\textwidth, height=.35\textwidth,viewport=0 0 530 400, clip]{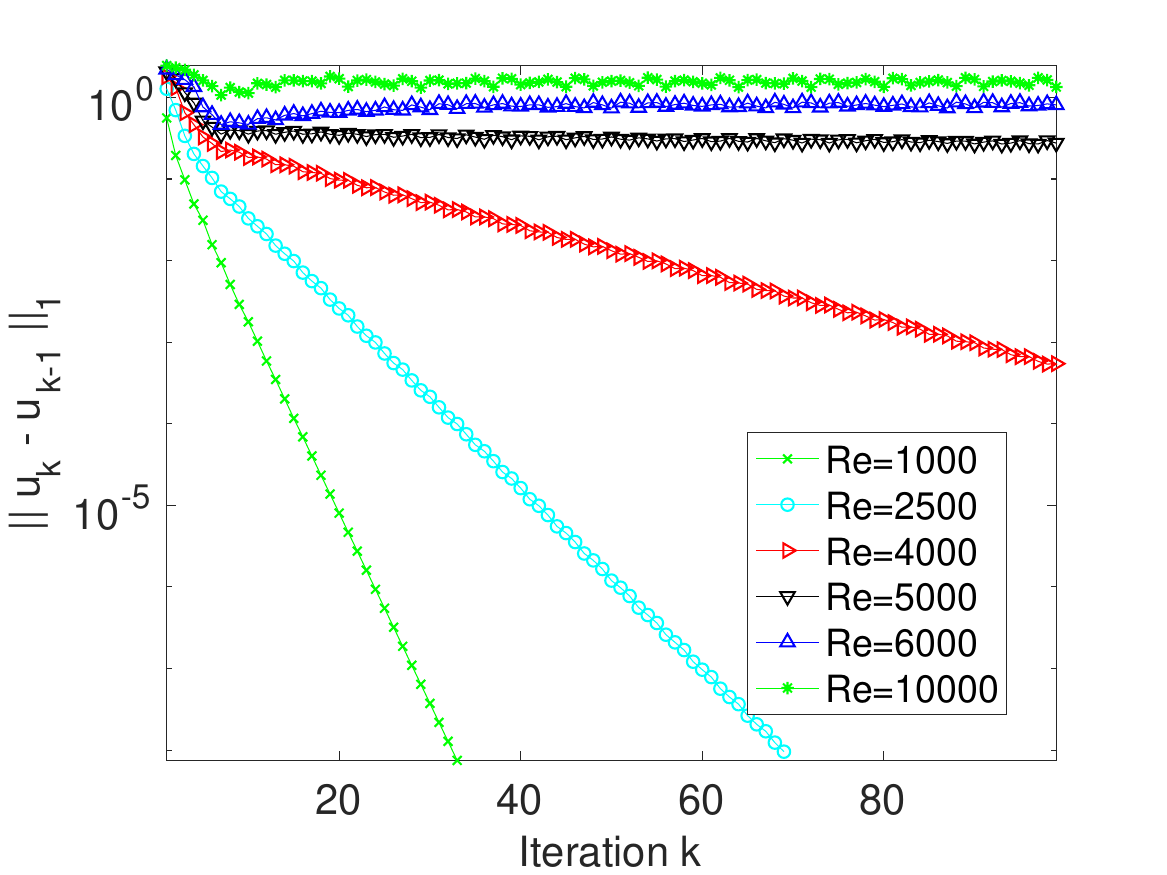}
\includegraphics[width = .4\textwidth, height=.35\textwidth,viewport=0 0 530 400, clip]{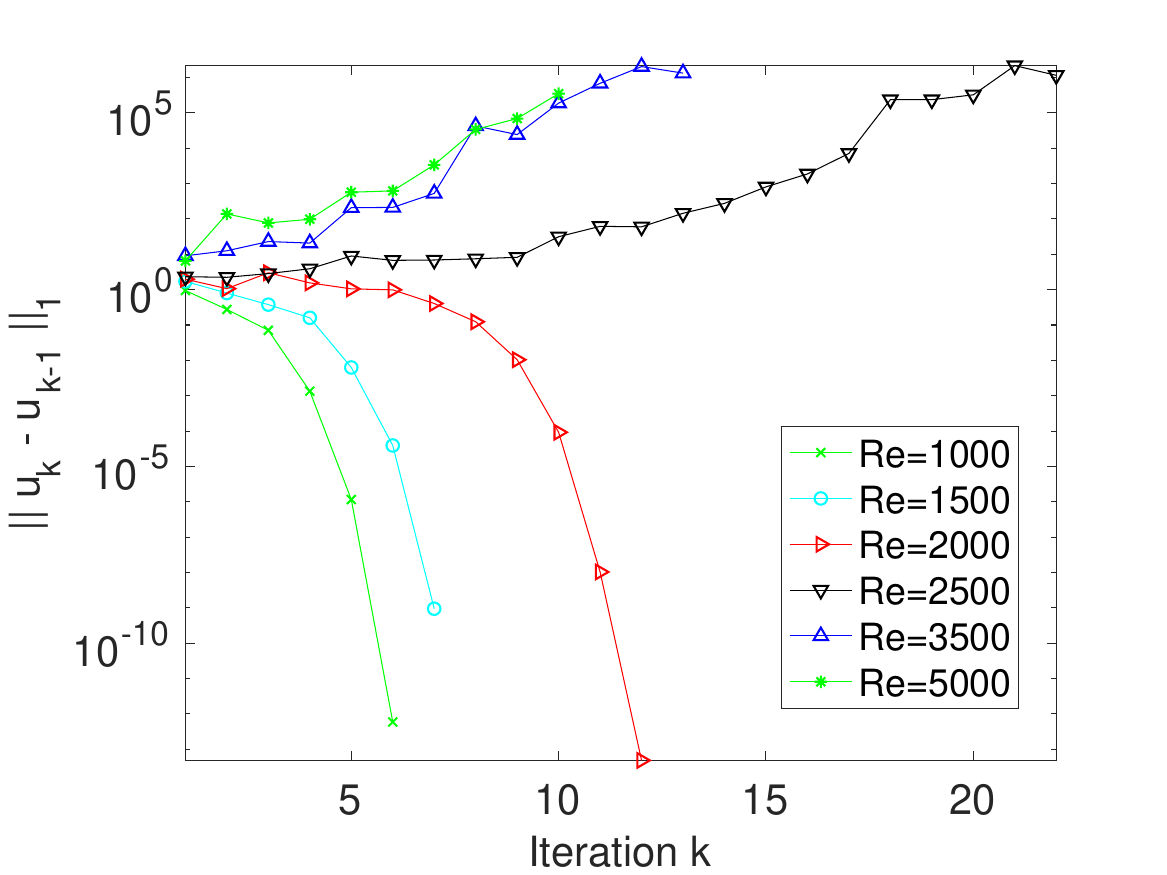}\\
Newton line search\\
\includegraphics[width = .4\textwidth, height=.35\textwidth,viewport=0 0 530 400, clip]{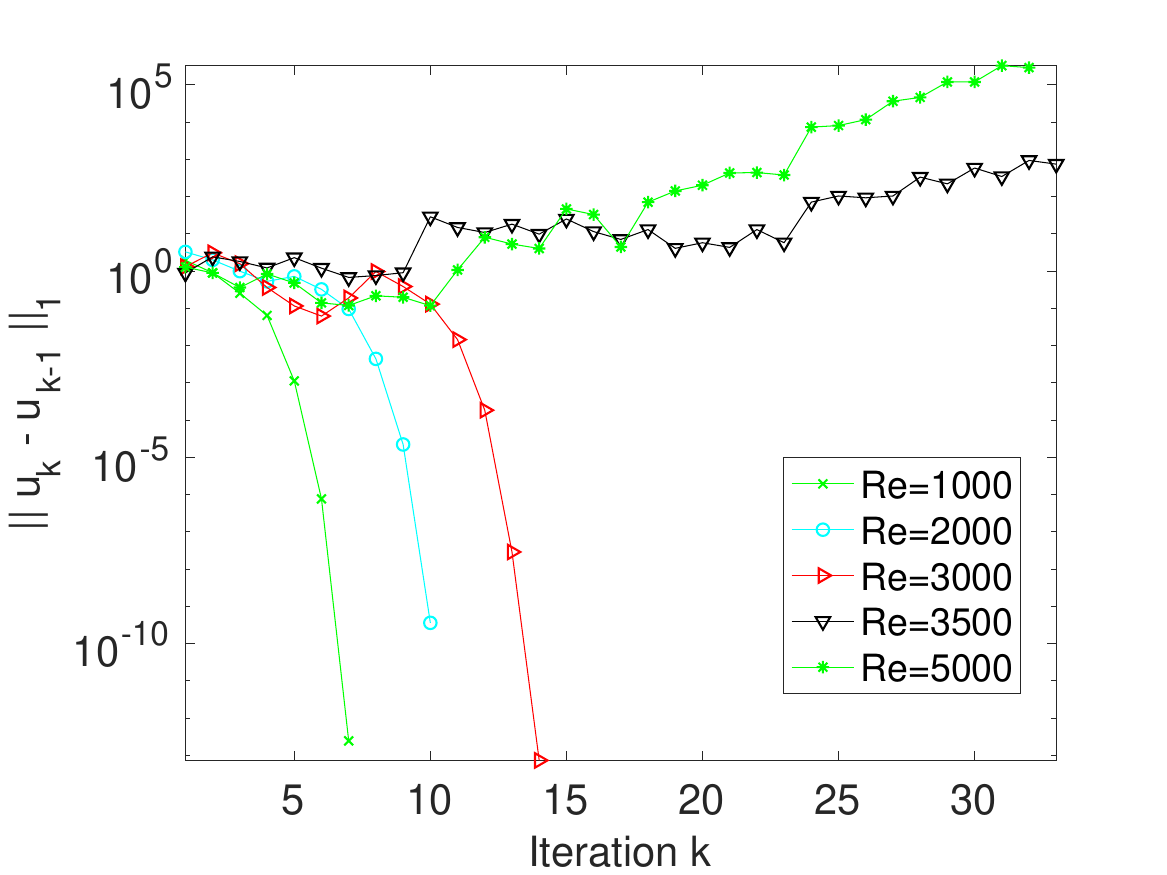}
\caption{\label{conv2} Shown above are convergence plots for the Picard (top left), Newton (top right) and Line Search Newton (bottom center), all using varying Re, but all using a barycenter refined $h=\frac{1}{64}$ mesh.}
\end{figure}

\subsubsection{3D driven cavity}

Our next test is for the 3D lid-driven cavity benchmark problem, which is a 3D analogue of the 2D cavity test problem above.  The domain is the unit cube, there is no forcing ($f=0$), homogeneous Dirichlet boundary conditions are enforced on the walls except that $u=\langle 1,0,0\rangle$ at the top of the box to represent the moving lid.  The viscosity is chosen as the inverse of the Reynolds number, and we will use several choices of $Re$.

\begin{figure}[ht]
\center
\includegraphics[width = .32\textwidth, height=.3\textwidth,viewport=0 0 550 415, clip]{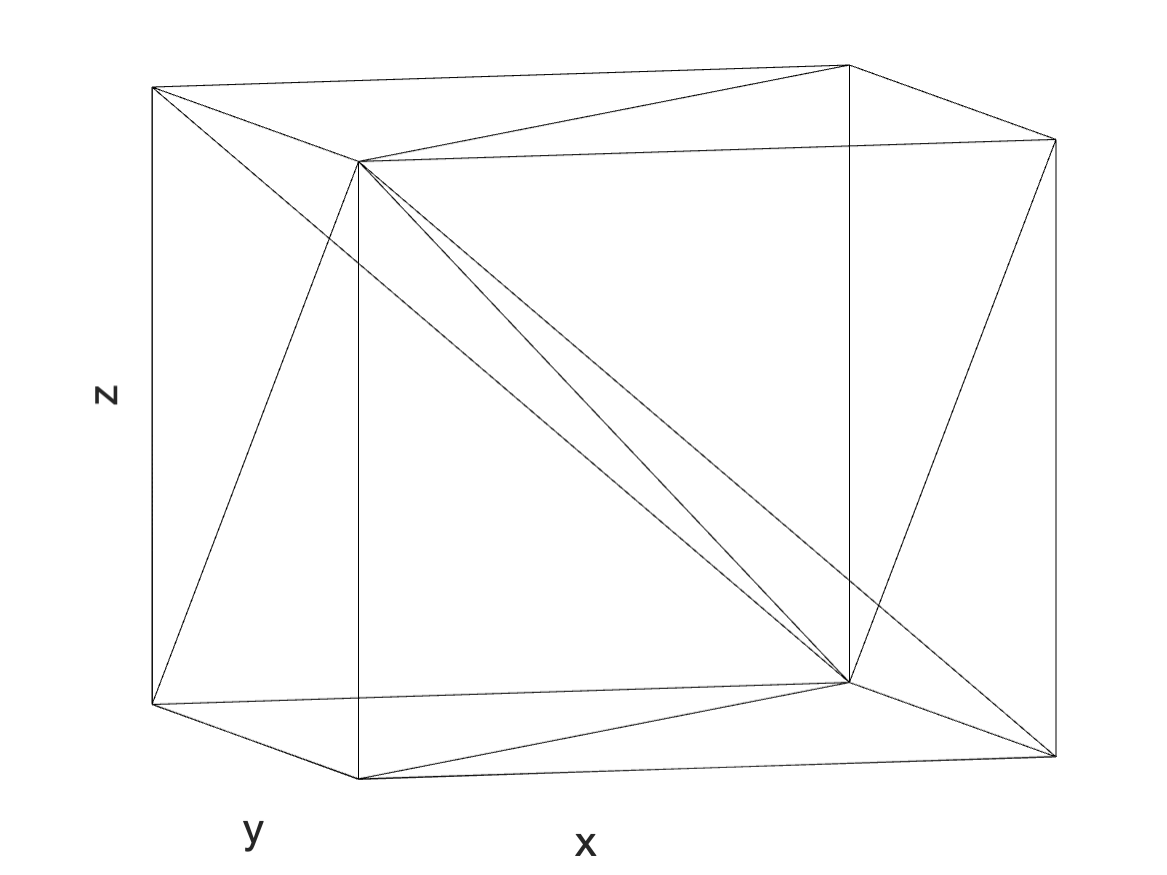}
\caption{\label{tet} Shown above is the method used to split a rectangular box into 6 tetrahedra.}
\end{figure}

We compute using two different meshes, which are constructed by first using Chebychev points on [0,1] to construct $\mathcal{M}\times \mathcal{M}\times \mathcal{M}$ grids of rectangular boxes (we use $\mathcal{M}$=8 for mesh 1, 12 for mesh 2 and $\mathcal{M}$=14 for mesh 3).  Each box is then split into 6 tetrahedra via the splitting shown in \ref{tet}, and then each of these tetrahedra is split into 4 tetrahedra with a barycenter refinement.  The mesh is then equipped with $(P_3, P_2^{disc})$ Scott-Vogelius elements, and this provides for approximately 206K total dof for mesh 1, 796K total dof for mesh 2, and 1.3 million total dof for mesh 3.  Solution plots found with this discretization matched those from the literature \cite{WongBaker2002}, and we show the midspliceplanes of the $Re$=1000 mesh 2 solution in figure \ref{fig:midsliceplanes}.

\begin{figure}[H]
			\centering
			\includegraphics[width = 0.95\textwidth, height=.27\textwidth,viewport=100 0 1100 300, clip]{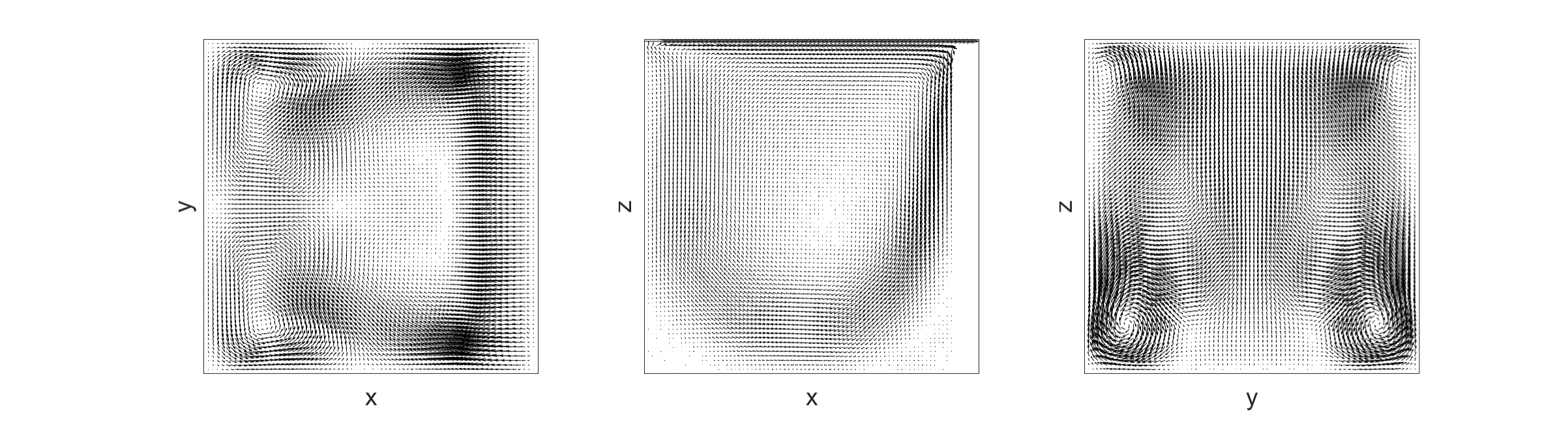}
			\caption{Shown above are midsliceplanes of the solution for the 3D driven cavity simulations at $Re$=1000.\label{fig:midsliceplanes}}
\end{figure}

We test convergence for the Newton, Picard, and Picard-Newton iterations for varying $Re$, and on all three meshes, always using $u_0=0$ for the initial iterate.  For each case, we test how many iterations it takes for the residual to drop below $10^{-8}$ in the $H^1$ norm.    
We consider failure to be lack of convergence within 200 iterations, 
and we call it blowup if the $L^2$ residual ever grows above $10^4$, 
(and when this happens our linear solvers fail).
\color{black}
Convergence results are shown in table \ref{conv10}.  As expected, Newton and Picard iterations only converge for smaller $Re$.  Picard-Newton, on the other hand, is able to converge for much larger $Re$, up to $Re=1800$ on mesh 3.  We note this is without any acceleration methods, continuation methods, line search, or other; it is simply the two step Picard-Newton iteration.  We note that the results suggest some mesh independence of the convergence: up to $Re=1400$, the number of iterates to convergence for Picard-Newton is about the same on meshes 2 and 3.  Interestingly, convergence for $Re$=1600 and 1800 is achieved on mesh 3 but not on mesh 1 or mesh 2.

{\footnotesize
\begin{table}[H]
\centering
\begin{tabular}{|c|c|c|c|c|c|c|c|c|c|c|c|c|}
	\hline
 &	& \multicolumn{3}{|c|}{Mesh 1}  &	& \multicolumn{3}{|c|}{Mesh 2} &	& \multicolumn{3}{|c|}{Mesh 3} \\ \hline
	$Re$ / method &  & Newt & Pic & P-N & &  Newt & Pic & P-N & &  Newt & Pic & P-N \\ \hline
100& &  	6& 30	& 4 &&	5   &  26 & 4  & & 5 & 26 & 4 \\ \hline
200&  & 	7& 121	& 5 &&	6   & 180 & 5 & & 6 & 185 & 5 \\ \hline
400&  & 	8& F	& 6 &&		8  &  F & 6 & & 8 & F & 6 \\ \hline
600&  & 	11& F	& 7 &&	14  &  F & 7 & & B & F & 7 \\ \hline
800&  & 	B& F	& 7 &&	B  &  F & 9 & & B & F & 9 \\ \hline
1000&  & 	B& F	& 9 &&	B  &  F & 11 & & B& F & 13 \\ \hline
1200&  & 	B& F	& 10 &&	B  &  F & 12 & & B & F & 13 \\ \hline
1400&  & 	B& F	& 10 &&	B  &  F & 34 & & B & F & 33 \\ \hline
1600&  & 	B& F	& 11 &&	B  &  F & F & & B & F & 68 \\ \hline
1800&  & 	B& F	& 11 &&	B  &  F & F & & B & F & 176 \\ \hline
2000&  & B	& F	& 13 &&	B  &  F & F & & B & F & F \\ \hline
\end{tabular}
\caption{Shown above are convergence results (number of iterations, `F' if no convergence after 200 iterations, `B' if $H^1$ residual grows above $10^4$) for the Picard, Newton and Picard-Newton iterations for varying $Re$ and on mesh 1 and mesh 2.\label{conv10}}
\end{table}
}

\section{Anderson accelerated Picard-Newton}

Anderson acceleration (AA) has recently been shown to be very effective at accelerating and enabling convergence of fixed point methods, and in particular for the Picard method for the NSE \cite{PRX19,EPRX20,PR21}.  Hence it seems appropriate to test if AA can also help improve convergence in the Picard-Newton iteration.  There are multiple ways AA could be applied: to the Picard iteration $g_P$, to the Newton iteration $g_N$, to the entire iteration $g_{PN}$, and combinations thereof.  In light of the recent work of \cite{X23,RX23} where it was shown that AA reduces the asymptotic convergence order of superlinearly converging methods including for the NSE, we focus here on applying AA only to the Picard step
(but we note that we did see modest improvement with AA applied to the $g_{PN}$; results omitted).  
%but we note that we did try AA applied to the $g_{PN}$ and say modest improvement; results omitted).  
For simplicity, we consider only the case of depth $m=1$ and no relaxation for these initial tests, but remark that further improvement is very likely possible with optimal depth and relaxation parameters.  Hence for the proposed AAPicard-Newton solver studied in this section, each iteration will require two Picard solves and one Newton solve at each iteration.

\subsection{AA background}
The Anderson acceleration algorithm is defined as follows.  For a fixed point function $g:X\rightarrow X$ with $X$ a Hilbert space with norm $\| \cdot \|_X$, the depth $m \ge 0$ AA algorithm with damping parameters $0 < \beta_{k+1} \le 1$ is given by:
%\color{blue}[Sara asks: if we're only doing $m=1$ with no damping, why not just show
%the algorithm for $m=1$ with no damping?]
\color{black}
\\ \ \\
Step 0: Choose $x_0\in X.$\\
Step 1: Find $w_1\in X $ such that $w_1 = g(x_0)-x_0$.  
Set $x_1 = x_0 + w_1$. \\
Step $k+1$: For $k=1,2,3,\ldots$ Set $m_k = \min\{ k, m\}.$\\
\indent [a.] Find $w_{k+1} = g(x_k)-x_k$. \\
\indent [b.] Solve the minimization problem for $\{ \alpha_{j}^{k+1}\}_{k-m_k}^k$
\begin{align}\label{eqn:opt-v0}
\min_{\sum_{j=k-m_k}^{k} \alpha_j^{k+1}  = 1} 
\left\| \sum_{j=k-m_k}^{k} \alpha_j^{k+1} w_{j+1} \right\|_X
\end{align}
\indent [c.] For a selected damping factor $0 < \beta_{k+1} \le 1$, set
\begin{align}\label{eqn:update-v0}
x_{k+1} =  \sum_{j= k-m_k}^k \alpha_j^{k+1} x_{j}
 + \beta_{k+1} \sum_{j= k-m_k}^k \alpha_j^{k+1} w_{j+1},
\end{align}
where $w_{k+1}:= g(x_k)-x_k$ represents the stage $k$ residual.  
 
 \begin{remark}
The convergence theory assumes the $\alpha_j^{k+1}$ are uniformly bounded.  As discussed in \cite{PR21,PR23}, this is equivalent to assuming full column rank of the matrix with columns $(w_{j+1}-w_j)_{j=k,k-1,...k-m}$ and can be controlled by length and angle filtering.
 \end{remark}

\color{black}

%
%
%There are three ways applying AA to Algorithm \ref{algpn}, since $g_P, g_N, g_{NP}$ are all fixed-point iterations: 
%\begin{enumerate}[(i)]
%\item Applying Anderson acceleration to the Picard iteration step only, i.e. $u_{k+1} = g_N (g_{AP}(u_k))$;
%\item Applying Anderson acceleration to the Newton's method only, i.e. $u_{k+1}  = g_{AN}(g_P(u_k))$;
%\item Applying Anderson acceleration to the Newton-Picard iteration from Algorithm \ref{algpn}, i.e. $u_{k+1} = g_{ANP}(u_k)$.
%\end{enumerate}
%We observe that case (i) performs best in numerical tests and the reason relies on two main results how AA affects the convergent behavior of a general fixed-point iteration $g$: Anderson acceleration 
%\begin{enumerate}[a)]
%\item improves linearly convergent iterative methods by reducing the contractive ratio, see \cite{PRX19, PR21}.
%\item decelerates quadratically convergent iterative methods by reducing the convergent order, see \cite{X23, RX23}.
%\end{enumerate}
%We present analytical discussion and explanation in details below.
%
%\subsection{Anderson acceleration for linearly convergent fixed-point iteration}
%

A key idea of \cite{PRX19,PR21} in understanding how AA improves linearly convergent methods was to define the gain of the optimization problem by
\begin{align}\label{thetadef}
\theta_k := \frac{ \left\| \sum_{j=k-m_k}^{k} \alpha_j^{k+1} w_{j+1} \right\|_X } { \| w_{k} \|_X }.
\end{align}
This is considered the gain factor because the numerator represents the minimum value of the sum using the optimal parameters, while the denominator is the value of the sum if  the usual fixed point method was used.  With this, it can be proven that AA improves the linear convergence rate by scaling it by the gain factor $\theta_k$ of the underlying AA optimization problem  \cite{PRX19,EPRX20,PR21} .  For AA depth $m=1$, the result from Theorem 4.1 in \cite{PR21} reads
\begin{align}\label{eqn:m1-thm}
&\nr{w_{k+1}}_X  \le
 \nr{w_k}_X \bigg\{
\theta_k \big((1-\beta_{k}) + \kappa_g \beta_{k} \big)
+ \hat \kappa_g \sigma^{-1} \sqrt{1-\theta_{k}^2}
\nonumber \\ & 
  \times\bigg( \nr{w_k}_X
  \Big(\sigma^{-1} \sqrt{1-\theta_{k}^2} + \beta_{k}\theta_{k} \Big)
  +\nr{w_{k-1}}_X
  \Big(\sigma^{-1} \sqrt{1-\theta_{k-1}^2} + \beta_{k-1}\theta_{k-1} \Big) 
  \bigg)  \bigg\},
\end{align}
where $\kappa_g$ is the linear convergence rate of the usual fixed point iteration, $\hat \kappa_g$ is the Lipschitz constant of $g'$, and $\sigma>0$ satisfies 
\begin{equation}
\| w_{k+1} - w_k \|_X \ge \sigma \| x_k - x_{k-1} \|_X. \label{sigdef}
\end{equation}
The above result assumes $\sigma>0 \ \forall k$ and $g$ is Lipschitz continuously differentiable, and we note that if $g$ is contractive then $\sigma = 1 - \kappa_g$ (\cite{PR21} discusses the case when $g$ is not contractive).  

%The result for general $m$ is analogous, and from Theorem 5.3 of \cite{PR21} we have that
%\begin{align}\label{eqn:tgenm}
%\nr{w_{k+1}}_X & \le \nr{w_k}_X \Bigg\{
% \theta_k ((1-\beta_{k}) + \kappa_g \beta_{k})
%+ C \hat \kappa_g 
%   \sum_{n = k-{m_{k-1}}}^{k} 
%\nr{w_n}_X
%  \Bigg\},
%\end{align}
%where $C$ depends on $\sigma^{-1}$, relaxation and gain parameters, as well as the degree to which the past $m$ differences $w_{j+1}-w_j$ are linearly independent.
\begin{remark}
Equation \ref{eqn:m1-thm} %\eqref{eqn:tgenm} 
also reveals whether the mesh size $h$ affects the convergence. For the NSE, the Lipschitz constant $\kappa_g$ is independent of $h$ \cite{PRX19}. We observe very little difference in the convergent plots for various $h$ in our tests. However, for certain fluid models, such as Bingham viscoplastic model, the Lipschitz constant $\kappa_g$ has a  negative scaling with $h$, see Lemma 3.5 in \cite{LRX23}. 
\end{remark}

\subsection{Numerical Tests for AAPicard-Newton}

We now give results for three numerical tests for AAPicard-Newton, with focus on how it improves over Picard-Newton.  For the AA, we use only depth $m=1$ and no relaxation.  We use the same linear solvers as in the tests above for Picard-Newton.

\begin{alg}[AAPicard-Newton with $m=1$ and no relaxation]
\label{alg:aapn}
The AAPicard-Newton iteration with Anderson depth $m=1$ and no relaxation consists of applying the composition of the Newton and Anderson accelerated Picard iterations for solving Navier-Stokes equations: $ g_N \circ g_{AP} $, i.e.,
\begin{enumerate}
\item[Step 1:] Find $\tilde u_{k+1} = g_P(u_{k})$ by finding $\tilde  u_{k+1}\in V$ 
satisfying 
\begin{equation}
\nu(\nabla \tilde u_{k+1},\nabla v) + b^*(u_k,\tilde u_{k+1},v) = (f,v) \ \forall v\in V. \label{wpaa}
\end{equation}
\item[Step 2:] Find $\tilde {\tilde u}_{k+1} = g_P(\tilde u_{k+1})$ by finding $\tilde  {\tilde u}_{k+1}\in V$ 
satisfying 
\begin{equation}
\nu(\nabla \tilde {\tilde u}_{k+1},\nabla v) + b^*(\tilde u_{k+1}, \tilde {\tilde u}_{k+1},v) = (f,v) \ \forall v\in V. \label{wpaa}
\end{equation}
\item[Step 3:] 
Determine $\alpha^{k+1}$ by
\[
\alpha^{k+1} = \frac{ (\nabla \tilde{\tilde u}_{k+1},\nabla (\tilde {\tilde u}_{k+1} - \tilde u_{k+1}) )}{\| \nabla  (\tilde {\tilde u}_{k+1} - \tilde u_{k+1}) \|^2}
\]
and set
\[
\hat u_{k+1} = (1-\alpha^{k+1})  \tilde {\tilde u}_{k+1} + \alpha^{k+1} \tilde u_{k+1}.
\]
\item[Step 4:] Find $u_{k+1} = g_N(\hat u_{k+1})$  by finding $u_{k+1}\in V$ satisfying  \eqref{wn}.
\end{enumerate}
\end{alg}

\color{black}

\subsubsection{2D Driven Cavity}
%This is the same 2D driven cavity problem from above, with the only changes here being the Reynolds number $Re$.
%Here we consider AA applied to $g = g_N \circ g_P$, and test it on Re=10000, 12000 and 15000 on a barycenter refined $\frac{1}{128}$ mesh using $m=0,1,2$ and relaxation.  Results are shown in figure \ref{convAA}, and we observe that AA can help with convergence in particular if relaxation $\beta=0.5$ is used; convergence was found in all tests where $m$=1 or 2, and $\beta=0.5$.  Without relaxation (i.e. $\beta=1$), convergence was sometimes achieved but not reliably; for example, for Re=10000 convergence was found for m=0 and m=1 but not m=2, and for Re=12000 convergence was found for m=0 and m=2 but not m=1.  We note that we tried other combinations of $m$ and $\beta$, but results were never significantly better than for small $m$ and $\beta=0.5$.
%
%\begin{figure}[ht]
%\center
%Re=10000 \hspace{1.2in} Re=12000 \hspace{1.2in} Re=15000 \\
%\includegraphics[width = .3\textwidth, height=.25\textwidth,viewport=0 0 530 400, clip]{Conv10000AA.pdf}
%\includegraphics[width = .3\textwidth, height=.25\textwidth,viewport=0 0 530 400, clip]{Conv12000AA.pdf}
%\includegraphics[width = .3\textwidth, height=.25\textwidth,viewport=0 0 530 400, clip]{Conv15000AA.pdf}
%\caption{\label{convAA} Shown above are convergence plots for AA applied to Picard-Newton on a barycenter refined $\frac{1}{128}$ mesh.}
%\end{figure}
We now test AAPicard-Newton on the 2D driven cavity problem (same setup as above) with Re=10000, 12000, 15000, 20000, and 25000, first with a barycenter refined $h=\frac{1}{128}$ uniform triangulation and also a barycenter refined $h=\frac{1}{196}$ uniform triangulation.  Plots of solutions for the $Re$=10000, 15000, 20000 and 25000 that were found on the $h=\frac{1}{196}$ mesh are shown in figure \ref{nseplot2}.  These plots match those found in \cite{erturk} quite well, although 25000 is beyond what is reported in their paper.  We note that no relaxation or continuation methods are used to get these results, and $u_0=0$ is always taken as the initial iterate.

\begin{figure}[ht]
\center
Re=10000 \hspace{1.3in} Re=15000 \hspace{1.3in} Re=20000 \\ % \hspace{.85in} Re=25000 \\
\includegraphics[width = .3\textwidth, height=.3\textwidth,viewport=115 45 460 390, clip]{Cavity10000.pdf}
\includegraphics[width = .3\textwidth, height=.3\textwidth,viewport=115 45 460 390, clip]{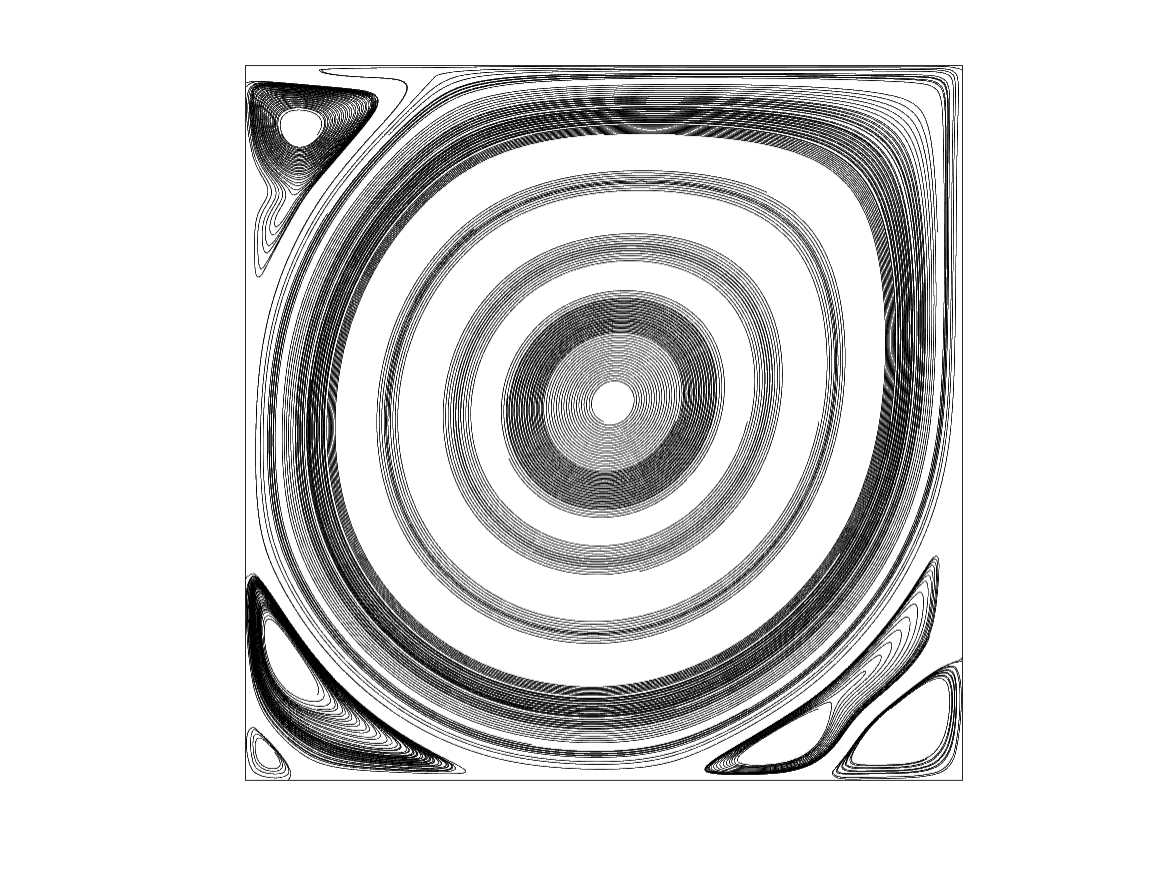}
\includegraphics[width = .3\textwidth, height=.3\textwidth,viewport=115 45 460 390, clip]{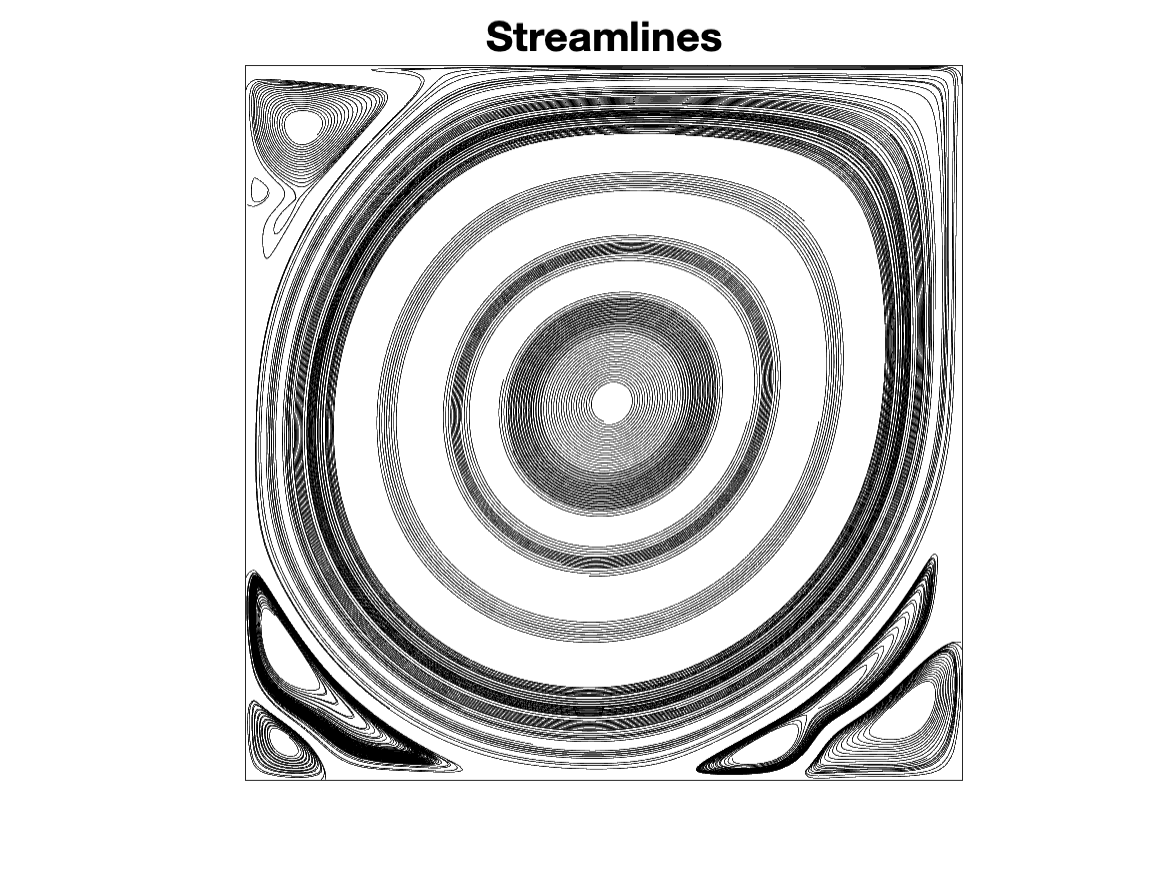}
\caption{\label{nseplot2} Shown above are streamlines of velocity solutions found for the 2D driven cavity problems with varying $Re$.}
\end{figure}

\begin{figure}[ht]
\center
\includegraphics[width = .4\textwidth, height=.35\textwidth,viewport=0 0 530 400, clip]{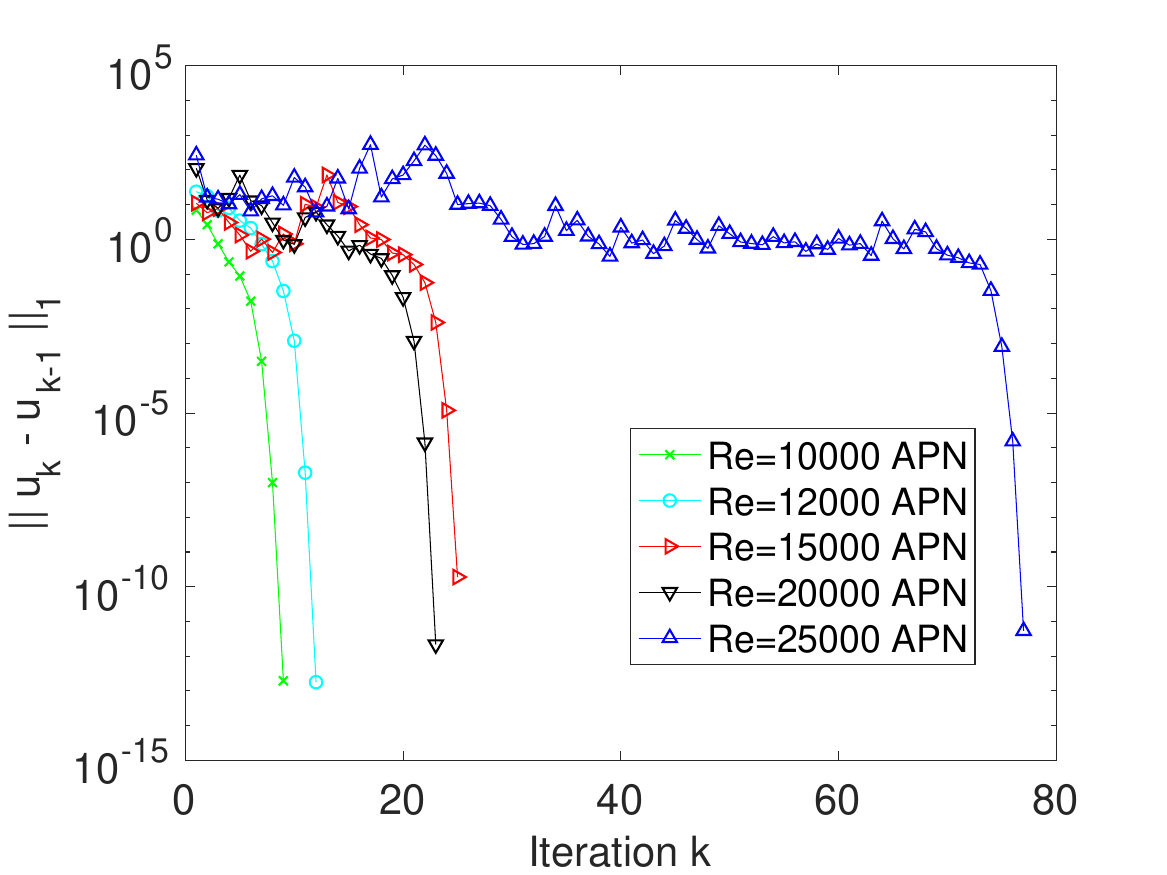}
\includegraphics[width = .4\textwidth, height=.35\textwidth,viewport=0 0 530 400, clip]{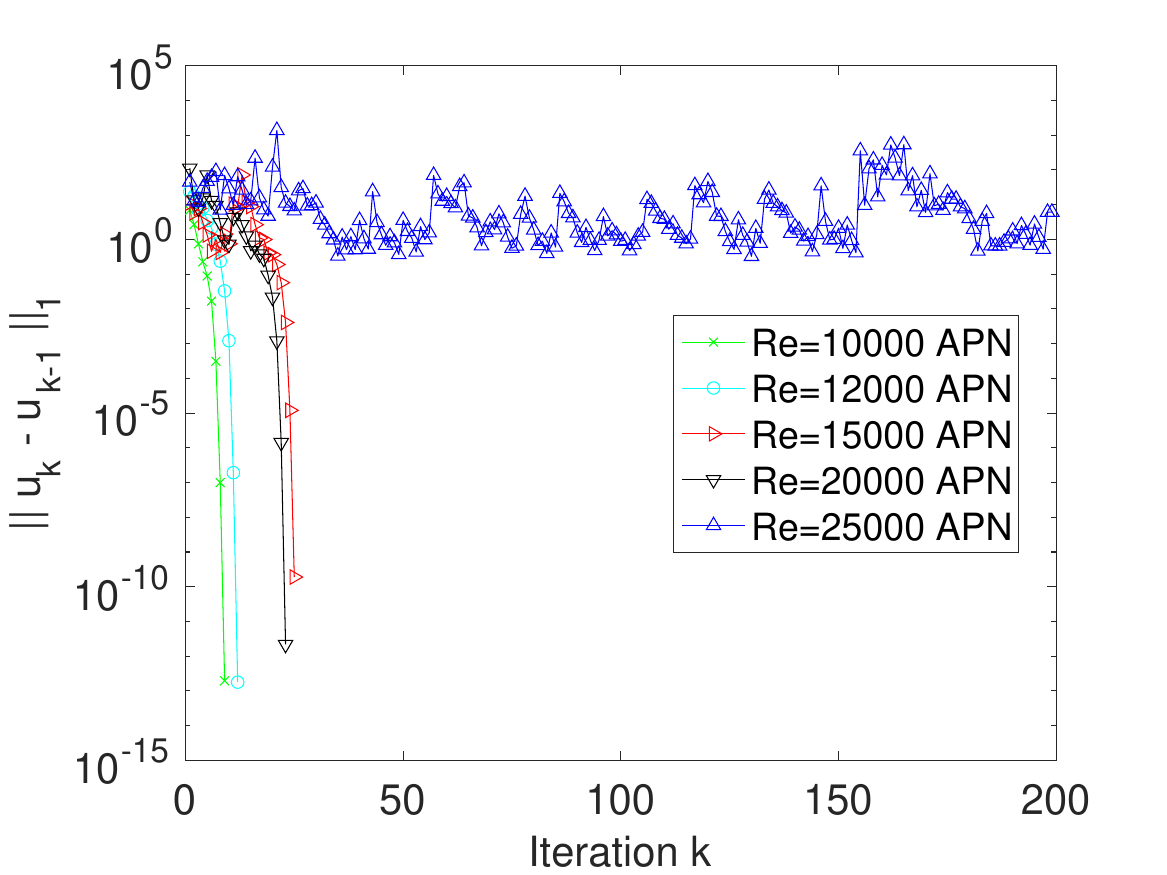}
\caption{\label{convAPN} Shown above are convergence plots for AAPicard-Newton on a barycenter refined $\frac{1}{128}$ (left) and  $\frac{1}{196}$ (right) mesh, for varying Re. 
}
\end{figure}

Convergence results for AAPicard-Newton are shown in figure \ref{convAPN}, and are quite good.  Convergence is achieved in all tests, and we note that there is little change between the convergence plots on the two meshes, suggesting that AAPicard is sufficiently preconditioning Newton so that convergence may be nearly mesh independent. Computing these solutions is well known to be quite difficult, as all are steady solutions at Reynolds numbers believed beyond the first bifurcation point where a natural flow would be periodic in time.  Many papers exist that compute steady solutions at higher Reynolds numbers (see e.g. \cite{BDD19,erturk} and references therein), however seemingly none of methods directly solve the steady problem on a single mesh as we do here and none are able to achieve convergence for such high $Re$; other successful methods use time-stepping (or pseudo time-stepping) and try to drive the problem to a steady state, and/or continuation methods and/or various multigrid approaches.

\subsubsection{Channel flow past a block}

The next experiment tests the proposed methods on channel flow past a block. Many studies can be found on this test, e.g. \cite{TGO15, R97, SDN99}. The domain consists of a $2.2\times0.41$ rectangular channel, and a block having a side length of $0.1$ with center at $(0.2, 0.2)$, with the origin set to be the bottom left corner of the rectangle, see figure \ref{fig:blockdomain}.
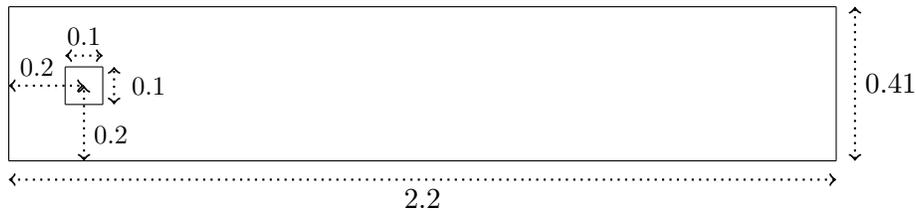
\begin{figure}[H]
	\center
	\begin{tikzpicture}[scale=5]
		\draw (0,0) -- (2.2,0);
		\draw [dotted, <->, thick] (0,-0.05) -- (2.2,-0.05);
		\draw (1.1,-0.05) node[below]{$2.2$};
		\draw (2.2,0) -- (2.2,0.41);
		\draw [dotted, <->, thick] (2.25,0) -- (2.25,0.41);
		\draw (2.25,0.205) node[right]{$0.41$};
		\draw (0,0.41) -- (2.2,0.41);
	%	\draw (1.9,0.30) node{outlet};
		\draw (0,0) -- (0,0.41);
		\draw[dotted, <->, thick] (0.2,0)--(0.2,0.2);
		\draw (0.15,0.15) -- (0.25,0.15);
		\draw (0.25,0.15) -- (0.25,0.25);
		\draw (0.15,0.25) -- (0.25,0.25);
		\draw (0.15,0.15) -- (0.15,0.25);
		\draw [dotted, <->, thick] (0,0.2)--(0.2,0.2);
	%	\draw (-0.2,0.3)node{inlet};
		\draw [dotted, <->, thick]  (0.28,0.15)--(0.28,0.25);
		\draw [dotted, <->, thick]  (0.15,0.28)--(0.25,0.28);
		\draw (0.3,0.2)node[right]{\small$0.1$};
		\draw (0.2,0.28)node[above]{\small$0.1$};
		\draw (0.075,0.20)node[above]{\small$0.2$};
		\draw (0.2,0.07)node[right]{\small$0.2$};
	\end{tikzpicture}
	\caption{Shown above is the domain for the channel flow past a block numerical experiment \label{fig:blockdomain}}
\end{figure}

\begin{figure}[ht]
\center
\includegraphics[width = .75\textwidth, height=.15\textwidth,viewport=100 160 490 240, clip]{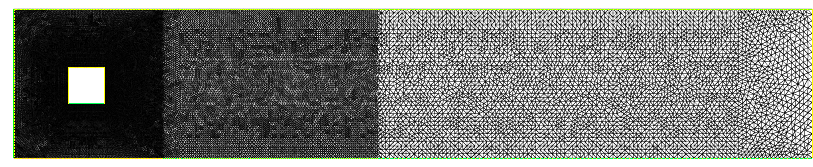}\\
\ \ \ \includegraphics[width = .75\textwidth, height=.14\textwidth,viewport=190 30 1400 250, clip]{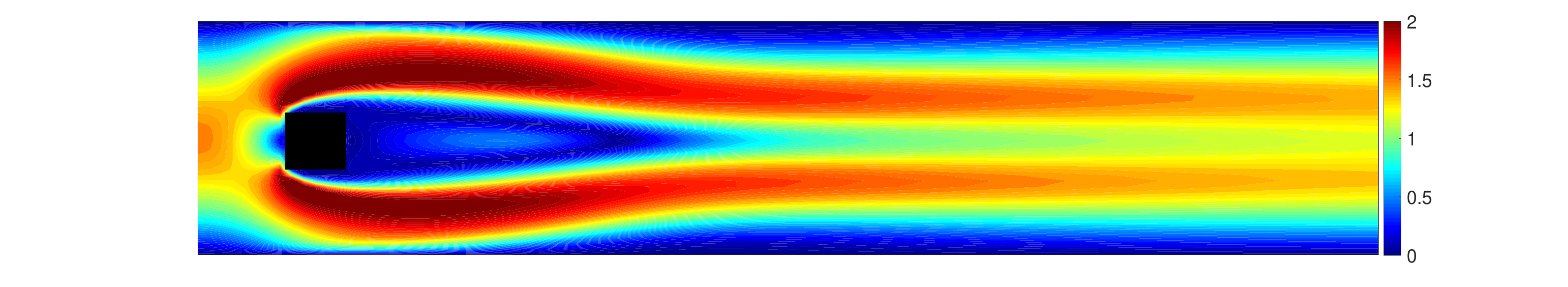}\\
\ \ \ \includegraphics[width = .75\textwidth, height=.14\textwidth,viewport=190 30 1400 250, clip]{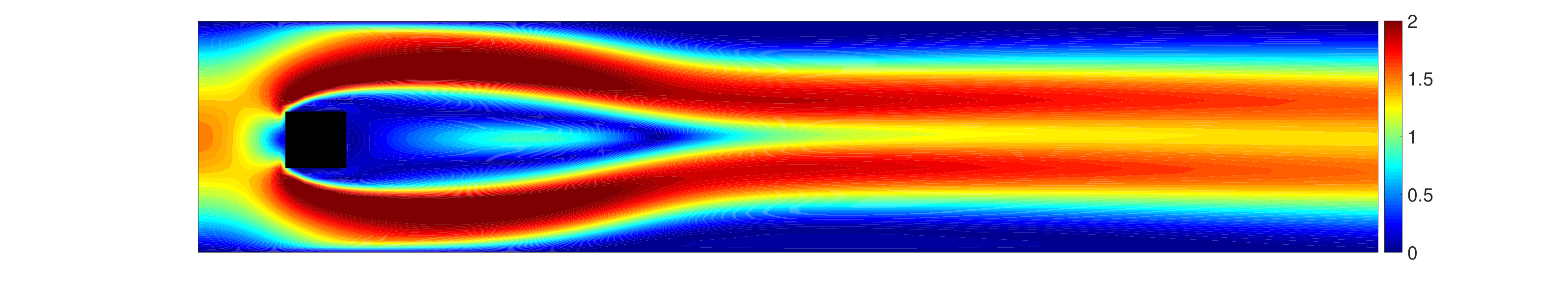}
\caption{\label{cylmesh} Shown above is the mesh used for the 2D flow past a block, before the barycenter refinement.}
\end{figure}

No-slip velocity boundary is imposed on the walls and block, and at the inflow the profile is set as 
\begin{align}
	u_1(0,y,t)=&u_1(2.2,y,t)=\frac{6}{0.41^2}y(0.41-y),\nonumber\\
	u_2(0,y,t)=&u_2(2.2,y,t)=0.
\end{align}
The outflow uses a zero traction (i.e. do nothing) homogeneous Neumann boundary condition.  This problem uses no external forcing, $f=0$. 
We compute with Reynolds numbers $Re=$100 and 150, which translate to $\nu=\frac{1}{1000}$ and $\frac{1}{1500}$ since the length scale defining the Reynolds number is taken to be the width of the block (L=0.1) for this problem.  We note that many numerical tests have been run using $Re=100$ and above for this test problem \cite{TGO15, R97, SDN99} and it is well known that this problem admits periodic in time solutions in this regime.  Still, we search for steady solutions at these Reynolds numbers.

We compute using a mesh refined towards the left half of the channel and then refined again around the block, and finally refined again with a barycenter refinement over the entire mesh.  The mesh before the barycenter refinement is applied is shown in figure \ref{cylmesh}.  We compute using $(P_2,P_1^{disc})$ Scott-Vogelius elements, which provides 832K velocity dof and 622 pressure dof.  

The Picard, Newton, Picard-Newton and AAPicard-Newton solvers are tested for both $Re=100$ and $Re=150$, and results are shown in figure \ref{cylconv}.  We observe that Newton fails in both cases, and in fact it essentially blows up (residual grows above $10^5$, at which point the simulation terminates).  Picard convergences linearly for $Re=100$ but fails for $Re=150$.  Picard-Newton converges quickly (and quadratically) in both cases, and AAPicard-Newton provides additional speedup.

\begin{figure}[ht]
\center
\includegraphics[width = .45\textwidth, height=.35\textwidth,viewport=0 0 550 400, clip]{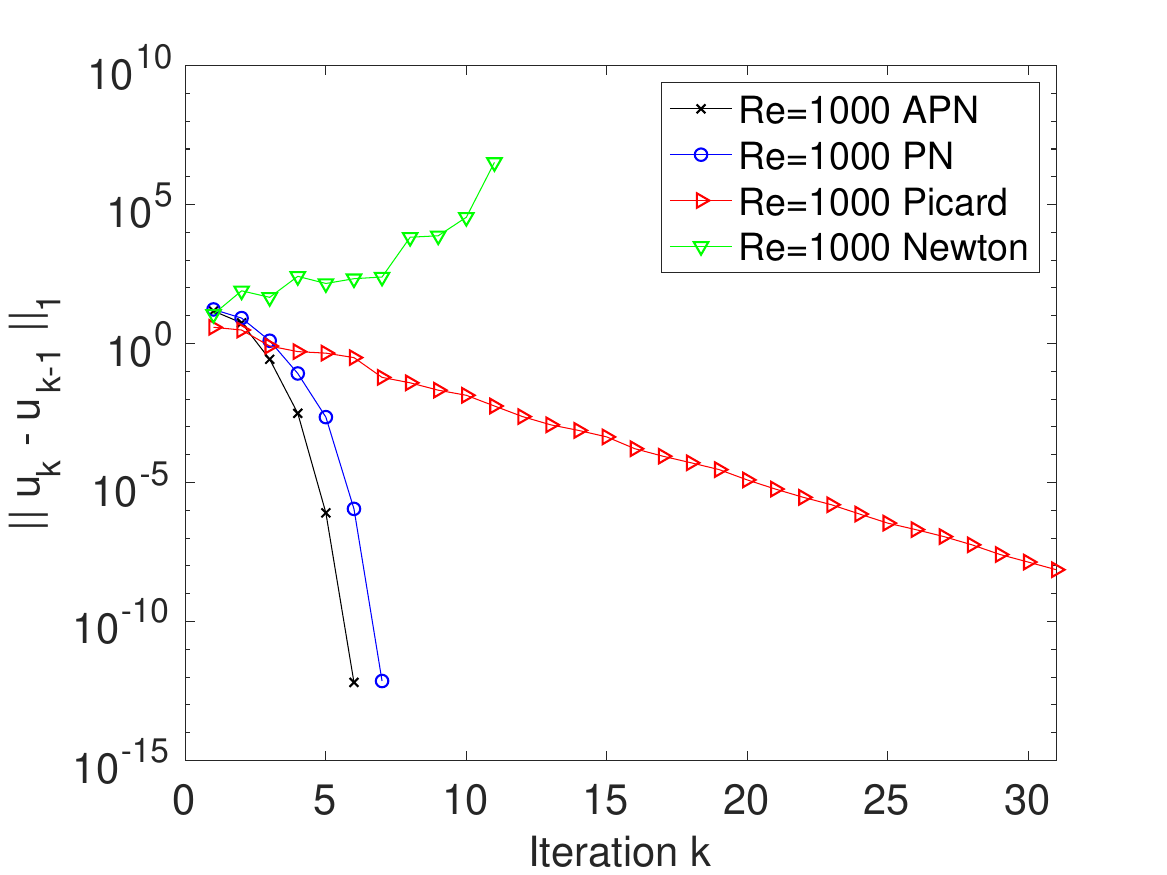}
\includegraphics[width = .45\textwidth, height=.35\textwidth,viewport=0 0 550 400, clip]{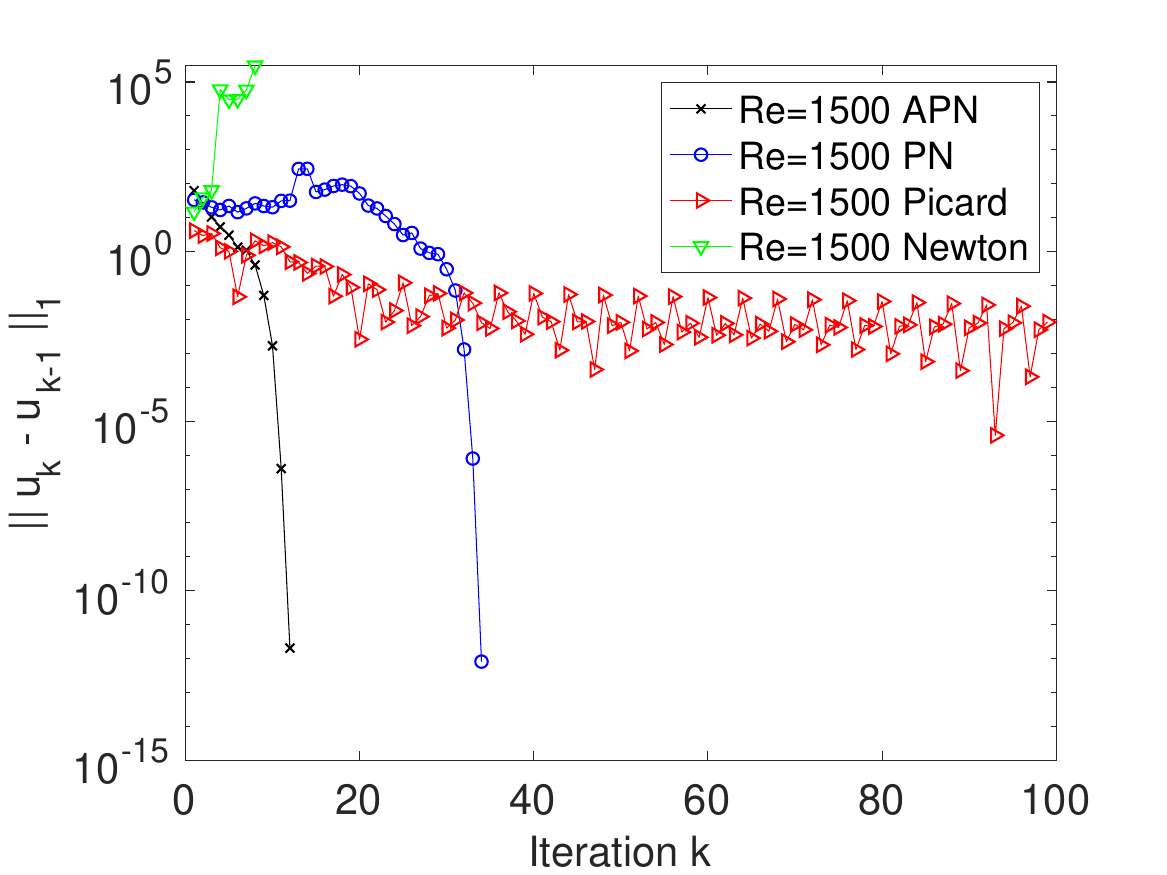}
\caption{\label{cylconv} Shown above is the convergence for Re=100 and Re=150 for channel flow past a block, for various solvers.}
\end{figure}

\subsubsection{3D driven cavity}

For our last test, we again use the 3D driven cavity benchmark problem.  We use the same problem setup as above, but now test AAPicard-Newton.  Results are in table \ref{conv11}, and we provide results from Picard-Newton (also shown above in table \ref{conv10}) for comparison.  Here we observe mixed results from using AA: it reduced the total number of iterations, and significantly, for $Re\le$ 1400, but is unable to converge for higher $Re$ than Picard-Newton does.  In fact, Picard-Newton converges at $Re$=1600 and 1800 on mesh 3 (but not meshes 1 or 2) while AAPicard-Newton does not converge on either mesh for these $Re$.  

{\footnotesize
\begin{table}[H]
\centering
\begin{tabular}{|c|c|c|c|c|c|c|c|c|c|c|}
	\hline
 &	& \multicolumn{2}{|c|}{Mesh 1}  &	& \multicolumn{2}{|c|}{Mesh 2} &	& \multicolumn{2}{|c|}{Mesh 3} \\ \hline
	$Re$ / method &  & P-N  & AAP-N & &  P-N & AAP-N & &  P-N & AAP-N \\ \hline
800&  &  7 & 6 &&  9 & 7 & & 9 & 7 \\ \hline
1000& & 9 & 6 && 11 & 8 & & 13 & 8 \\ \hline
1200& & 10 & 7 && 12 & 10 & & 13 & 10 \\ \hline
1400& & 10 & 7 && 34 & 19 & & 33 & 15 \\ \hline
1600& & 11 & 8 && F & F & & 68 & F \\ \hline
1800& & 11 & 8 && F & F & & 176 & F \\ \hline
\end{tabular}
\caption{Shown above are convergence results (number of iterations, `F' if no convergence after 200 iterations, `B' if $H^1$ residual grows above $10^4$) for the  Picard-Newton and AAPicard-Newton iterations for varying $Re$ and on meshes 1,2, and 3.\label{conv11}}
\end{table}
}

\section{Conclusions and Future Directions}

We have analyzed and tested the Picard-Newton iteration for the incompressible Navier-Stokes.  We have proven that this simple to implement method is locally quadratically convergent for any problem data, and has stability properties not shared by the Newton iteration: it is globally stable for problem data satisfying the uniqueness condition, and is stable at each half step for general data.  Numerical tests show it is very effective on several benchmark tests, exhibiting quadratic convergence for worse initial iterates and for much higher Reynolds number than usual Newton or Picard.  Moreover, even when Picard-Newton did not converge, it remained stable (as opposed to Newton, which typically blows up when it does not converge).  We also considered Picard-Newton with the Picard step enhanced with Anderson acceleration in its simplest case (depth 1 and no relaxation).   Numerical results for AAPicard-Newton showed improvement over Picard-Newton, and showed a remarkable ability to converge for much higher $Re$ than typical solvers.

Future directions include further study of Anderson acceleration applied to Picard-Newton (both application of AA to only the Picard step and to the Picard-Newton iteration itself), as well as extending the ideas to other nonlinear PDEs.

\color{black}

\section{Acknowledgements}

This material is based upon work supported by the National Science Foundation under Grant No. DMS-1929284 while the authors were in residence at the Institute for Computational and Experimental Research in Mathematics in Providence, RI, during the Acceleration and Extrapolation Methods (SP, LR and XT), and the Numerical PDEs: Analysis, Algorithms and Data Challenges (SP and LR), programs.

\section{Declarations}

{\bf Funding}  Author SP acknowledges support from the National Science Foundation from grant DMS-2011519.  Author LR acknowledges support from the National Science Foundation from grant DMS-2011490.
   \\ \ \\
{\bf Conflict of interest} The authors have no conflicts of interest to declare that are relevant to the content of this article. The authors have no relevant financial or non-financial interests to disclose.\\ \ \\
{\bf Code Availability} The code for the current study is available from the corresponding author on reasonable request.\\ \ \\
{\bf Availability of Data and Materials} The datasets generated during and/or analyzed during the current study are available from the corresponding author on reasonable request.

\bibliographystyle{plain}
%\bibliography{graddiv,Xuejian_Li_ref}

\end{document}